\documentclass[a4paper,12pt]{article}

\usepackage{amsmath,amsthm,amssymb,enumerate}
\usepackage{xcolor}
\usepackage{float}
\usepackage{multirow}
\usepackage{longtable}
\usepackage{diagbox}
\usepackage{soul}
\usepackage{cancel}
\usepackage{gensymb}
\usepackage{enumitem}
\usepackage[square,sort&compress,comma,numbers]{natbib}
\usepackage{bm}
\usepackage{hyperref}
\usepackage{subfig}
\usepackage{graphicx}
\usepackage[margin=1.7cm]{geometry}
\usepackage{tikz}
\usepackage{esint}
\usepackage{caption}
\usetikzlibrary{shapes,calc}
\usepackage{verbatim}
\usepackage{array}
\usepackage{bm}
\definecolor{maroon}{RGB}{144,0,32}
\newcolumntype{H}{>{\setbox0=\hbox\bgroup}c<{\egroup}@{}}

\usepackage{newtxtext,newtxmath}

\usepackage{multirow}
\usepackage{marginnote}
\chardef\bslash=`\\ 

\newtheorem{thm}{Theorem}[section]
\newtheorem{cor}[thm]{Corollary}
\newtheorem{lem}[thm]{Lemma}

\newtheorem{rem}[thm]{Remark}

\theoremstyle{definition}
\theoremstyle{remark}

\numberwithin{equation}{section}

\newcommand{\bI}{\boldsymbol{I}}

\newcommand{\cA}{\mathcal A}
\newcommand{\cK}{\mathcal K}

\newcommand{\cD}{\mathcal D}
\newcommand{\cE}{\mathcal E}

\newcommand{\cT}{\mathcal{T}}

\newcommand{\cN}{\mathcal N}

\newcommand{\cM}{{\mathcal M}}
\newcommand{\err}{{\rm err}}

\newcommand{\hc}{\widehat{c}}

\newcommand{\divc}{\mathrm{div}}
\newcommand{\rot}{\mathrm{rot\,}}
\newcommand{\dof}{\mathrm{dof}}

\newcommand{\fl}{\,\, \text{for all}\:}

\newcommand{\half}{\frac{1}{2}}

\newcommand{\dx}{{\rm\,dx}}
\newcommand{\ds}{{\rm\,ds}}

\newcommand{\Poincare}{Poincar\'e}
\newcommand{\Holder}{H\"{o}lder~}

{\algorithm}%
{\endalgorithm}

\theoremstyle{definition}

\numberwithin{equation}{section}

\newcommand{\bV}{\text{\bf V}}

\newcommand{\bv}{\boldsymbol{v}}

\newcommand{\N}{\mathbb{N}}

\newcommand{\T}{\mathcal{T}}
\renewcommand{\P}{\mathcal{P}}

\newcommand{\bx}{{\boldsymbol{x}}}
\newcommand{\bn}{{\boldsymbol{n}}}
\newcommand{\bu}{{\boldsymbol{u}}}
\newcommand{\bp}{{\boldsymbol{p}}}
\newcommand{\bsf}{{\boldsymbol{f}}}
\newcommand{\bPi}{{\boldsymbol{\Pi}}}
\newcommand{\bpsi}{{\boldsymbol{\psi}}}

\def \R{{{\Bbb R}}}
\def \P{{{\mathcal P}}}

\allowdisplaybreaks

\def\R{\mathbb{R}}

\def\cA{\mathcal{A}}

\def\O{\Omega}

\def\bv{{\boldsymbol v}}

\def\pw{\rm {pw}}

\def\cE{{\mathcal{E}}}

\def\jump#1{\left[\hskip -3.5pt\left[#1\right]\hskip -3.5pt\right]}
\def\bchi{\boldsymbol{\chi}}

\def\bg{\boldsymbol \gamma}

\newcommand{\be}{\begin{equation}}
\newcommand{\ee}{\end{equation}}

\usepackage{mathtools}  
\mathtoolsset{showonlyrefs} 
\usepackage[normalem]{ulem}
\definecolor{violet}{rgb}{0.580,0.,0.827}

\usepackage[normalem]{ulem}
\newcounter{corr}
\definecolor{violet}{rgb}{0.580,0.,0.827}
\newcommand{\corr}[3]{\typeout{Warning : a correction remains in page
		\thepage}
	\stepcounter{corr}        
	{\color{red}\ifmmode\text{\,{\ensuremath{#1}}\,}\else{#1}\fi}
	{\color{blue}#2}
	{\color{violet} #3}}

\newcounter{changeto}
\newcommand{\changeto}[2]{\typeout{Warning : a correction remains in page
		\thepage}
	\stepcounter{changeto}        
	{\color{blue}\ifmmode\text{\,\sout{\ensuremath{#1}}\,}\else\sout{#1}\fi}
	{\color{red}#2}}

\title{Nonconforming virtual element method for an incompressible miscible displacement problem in porous media}
\author{Sarvesh Kumar \footnote{Department of Mathematics, Indian Institute of Space Science and Technology, Thiruvanathapuram 695547, India. sarvesh@iist.ac.in} \quad Devika Shylaja \footnote{Department of Mathematics, Indian Institute of Space Science and Technology, Thiruvanathapuram 695547, India. devikas.pdf@iist.ac.in}}

\date{}
\begin{document}
	\maketitle
	
\abstract This article presents a priori error estimates of the miscible displacement of one incompressible fluid by another through a porous medium characterized by a coupled system of nonlinear elliptic and parabolic equations. The study utilizes the $H(\divc)$ conforming virtual element method (VEM) for the approximation of the velocity, while a non-conforming virtual element approach is employed for the concentration. The pressure is discretised using the standard piecewise discontinuous polynomial functions. These spatial discretization techniques are combined with a backward Euler difference scheme for time discretization. The article also includes numerical results that validate the theoretical estimates presented.



	\section{Introduction}

The {\textbf{miscible displacement}} of one incompressible fluid by another through a porous
medium is described by a time-dependent coupled system of nonlinear partial differential equations \cite{miscibledisplacement_1983,chavent_jaffre,peaceman-77}. In this process, two fluids that are capable of mixing evenly (miscible) displace each other within the interconnected void spaces of a porous medium, typically a rock formation. Let $\Omega\subset \R^2$ be a polygonal bounded, convex domain, describing a reservoir of unit thickness. Given a time interval $J := [0,T],$ for $T > 0$, the problem is to find the Darcy velocity $\bu=\bu(\bx,t)$ of the fluid mixture, the pressure $p = p(\bx,t)$ in the fluid mixture, and the concentration $c=c(\bx,t)$ of one of the components in the mixture, with $(\bx,t)=\Omega_T:=\Omega \times J$ such that
\begin{subequations}\label{eqn.model}
\begin{align}
&\displaystyle\phi\frac{\partial c}{\partial t}+\bu\cdot \nabla c -\divc(D(\bu)\nabla c)=q^{+}(\hc-c)\label{eqn.modela}\\
&\divc\,\bu=G\label{eqn.modelb}\\
&\bu=-a(c)(\nabla p-\bg(c)),\label{eqn.modelc}
\end{align}
\end{subequations}
where $\phi=\phi(\bx)$ is the porosity of the medium, $\bg(c)$ describes the force density due to gravity, and $a(c)=a(c,\bx)$ is the scalar-valued function given by
\[a(c):=\frac{k}{\mu(c)}.\]
Here, $k=k(\bx)$ represents the permeability of the porous rock, and $\mu(c)$ is the viscosity of the fluid mixture. Further, the non-negative injection and production source terms are $q^+=q^+(\bx,t)$ and $q^-=q^-(\bx,t)$ respectively, $\hc=\hc(\bx,t)$ is the concentration of the injected fluid, and
\begin{equation}\label{defn.G}
G:=q^+-q^-.
\end{equation}
Moreover, the diffusion dispersion tensor $D(\bu) \in \R^{2 \times 2}$ is given by
\begin{equation}\label{defn.D}
D(\bu):=\phi[d_m\bI +|u|\left(d_\ell E(\bu)+d_t(\bI-E(\bu))\right)]
\end{equation}
where $d_m$ is the molecular diffusion coefficient, $d_\ell$ (resp. $d_t$) is the longitudinal (resp. transversal) dispersion coefficient, $\bI$ is the identity matrix of order 2, and $E(\bu)$ is the tensor that projects onto $\bu$ direction which is given by, for $\bu=(u_1,u_2)$,
\begin{align*}
E(\bu)&:=\frac{\bu\bu^T}{|\bu|^2},\quad |\bu|^2=u_1^2+u_2^2.
\end{align*}

\medskip

\noindent Assume that no flow occurs across the boundary $\partial \Omega$, that is,
\begin{subequations}\label{eqn.bc}
\begin{align}
&\bu\cdot \bn=0 \quad \mbox{on }\partial \Omega \times J,\label{eqn.bc1}\\
&D(\bu)\nabla c\cdot \bn=0 \quad \mbox{ on } \partial \Omega \times J,\label{eqn.bc2}
\end{align}
\end{subequations}
where $\bn$ denotes the outward unit normal to the boundary $\partial \Omega$ and the initial condition
\begin{equation}\label{eqn.ic}
c(\bx,0)=c_0(\bx) \mbox{ in } \Omega,
\end{equation}
where $0 \le c_0(\bx) \le 1$ represents the initial concentration. A use of the divergence theorem for \eqref{eqn.modelb}, \eqref{defn.G}, and \eqref{eqn.bc1} shows the following compatibility conditions for $q^+$ and $q^-$:
\[\int_\Omega q^+(\bx,t)\dx=\int_{\Omega}q^-(\bx,t)\dx.\]
Since the pressure $p$ in \eqref{eqn.modelc} is only determined up to an additive constant, to ensure a unique solution, we make the assumption that
\[\int_\O p(\bx,t)\dx=0 \quad \fl t \in (0,T).\]

\smallskip

\noindent The study of miscible displacement is crucial for optimizing various industrial processes, such as enhanced oil recovery in the petroleum industry or contaminant transport in environmental remediation \cite{RussellWheeler,peaceman-77}. Understanding the underlying physics and employing accurate numerical simulations contribute to the development of efficient strategies for fluid displacement in porous media, thereby enhancing resource recovery and environmental management. The numerical methods to approximate the miscible displacement processes has been studied using finite difference methods in \cite{Douglas_1983,Peaceman_1966,peaceman-77}, finite element methods (FEMs) in \cite{miscibledisplacement_1983,Wheeler_1980, Wheeler_IP_1980, Wheeler_1982}, discontinuous Galerkin FEMs in \cite{BartelsJensenMuller_dG_2009,RiviereWalkington_2011,SunRiviereWheeler_2002,WangZhengYuGuoZhang_2019,GuoYuYang_2017}, finite volume methods in \cite{cha-07-por,ABEOM_2008,CHCKSMA_2013} and so on. 

\smallskip

\noindent {\textbf{Virtual element method}} (VEM) \cite{Veiga_basicVEM}, which is a generalization of the FEM,  has got more and more attention in recent years, because it can deal with the polygonal meshes and avoid an explicit construction of the discrete shape function, see \cite{Veiga_HdivHcurlVEM,Veiga_hitchhikersVEM,Brenner_errorVEM,VEM_general_2016,MixedVEM_general_2016,Cangiani_2017} and the references therein. The polytopal meshes can be very useful for a wide range of reasons, including meshing of the domain (such as cracks) and data features, automatic use of hanging nodes, adaptivity. Recently, a virtual element method for complex fluid flow problems \eqref{eqn.model}-\eqref{eqn.ic} has been investigated in \cite{Veiga_miscibledisplacement_2021}, where conforming VEM is analysed for the concentration and mixed VEM for the velocity-pressure equations. In contrast to this, the present work explores the numerical approximation of concentration using a nonconforming VEM \cite{ncVEM_2016,Cangiani_2017} (for any order of accuracy) and thus, introduces a novel perspective into the numerical analysis. Nonconforming methods typically impose fewer restrictions on the mesh topology. This can simplify the meshing process and reduce the effort required for mesh generation. Conforming methods, on the other hand, often demand a more regular mesh to satisfy certain continuity conditions. Nonconforming VEM allows discontinuities at element boundaries and this flexibility in the continuity conditions can be advantageous in handling irregular meshes and for solving reaction-dominated problems \cite[Section 9.1]{Cangiani_2017}. Additionally, an algebraic equivalence between the nonconforming VEM and a family of mimetic finite difference methods \cite{LipnikovManzini_mimetic} is established in \cite{ncVEM_2016}.

\smallskip

\noindent This paper employs the H(div) conforming VEM for approximation of the velocity, while the concentration is handled using a non-conforming virtual element approach. To discretize the pressure, standard piecewise discontinuous polynomial functions are used. These spatial discretizations are then combined with an uncomplicated time discretization using a backward Euler method, known for its computational efficiency. Optimal a priori error estimates are established for the concentration, pressure, and velocity in $L^2$ norm under regularity assumption on the exact solution. Numerical results are presented to justify the theoretical estimates.

\smallskip

\noindent The remaining parts are organised as follows. Section \ref{sec:wf} discusses the weak formulation of \eqref{eqn.model}-\eqref{eqn.ic}. The main result of this paper is stated at the end of this section. Section~\ref{sec:vem} deals with the virtual element method, semi-discrete and fully-discrete formulations. Error estimates are established in Section~\ref{sec:error}. Section~\ref{sec.numericalresults} provides the results of computational experiments that validate the theoretical estimates on both an ideal test case and a more realistic test case. The paper ends with an appendix, Section \ref{sec:appendix}, where we prove the error estimate for the concentration, stated in Theorem~\ref{thm.c}.

\smallskip

\noindent {\textbf{Notation.}} The standard $L^2$ inner product and norm on $L^2(\O)$ are denoted by $(\cdot,\cdot)$ and $\|{\cdot}\|$. The semi-norm and norm in $W^{k,p}(D)$,  for $D \subseteq \Omega$ and $1 \le p \le \infty$, are denoted by $|\bullet|_{k,p,D}$ and $\|\bullet\|_{k,p,D}$. For $p=2$,  the semi-norm and norm are denoted by $|\bullet|_{k,D}$ and $\|\bullet\|_{k,D}$. Let $\P_k(D)$ denotes the space of polynomials of degree at most $k$ ($k \in \N_0$) with the usual convention that $\P_{-1}(D)=0$.

\smallskip

 \noindent Let $H(\divc,\Omega) $ denotes the Sobolev space
\[H(\divc,\Omega):=\{\bv \in (L^2(\Omega))^2: \divc \bv \in L^2(\O)\}.\] Define the velocity space $\bV$, the pressure space $Q$, and the concentration space $Z$, equipped with the following norms by
\begin{align}\label{defn.spaces}
	&\bV:=\{\bv \in H(\divc,\Omega):\bv \cdot \bn=0 \mbox{ on }\partial \Omega\},\quad \|\bu\|_\bV^2:=\|\bu\|^2+\|\divc \bu\|^2\\
	&Q:=L^2_0(\Omega):=\{q \in L^2(\Omega):(q,1)=0\}, \quad \|q\|_Q^2:=\|q\|^2\\
	&Z:=H^1(\Omega), \quad \|z\|_Z^2:=\|z\|^2+\|\nabla z\|^2.
\end{align}
For $0 \le a \le b$,
\[\|\bv\|_{L^2(a,b;\bV)}^2:=\int_a^b\|\bv(t)\|_\bV^2 \dx, \quad \|\bv\|_{L^\infty(a,b;\bV)}:=\mbox{ess} \sup_{t \in [a,b]}\|\bv(t)\|_\bV.\]

\smallskip

\noindent For all $s>0$, define the broken Sobolev space as
\[H^s(\cT_h):=\{v \in L^2(\O);\,v_{|K} \in H^s(K) \fl K \in \cT_h\},\]
with the corresponding broken semi-norms and norms
\[|v|_{s,\cT_h}^2:=\sum_{K \in \cT_h}|v|_{s,K}^2, \quad \|v\|_{s,\cT_h}^2:=\sum_{K \in \cT_h}\|v\|_{s,K}^2.\]

\section{Weak formulation}\label{sec:wf}
This section deals with the weak formulation of the continuous problem \eqref{eqn.model}-\eqref{eqn.ic} and the properties of the associated bilinear forms.

\medskip


\noindent Assume that the functions $a$ and $\phi$ in \eqref{eqn.model} are positive and uniformly bounded from below and above, i.e, there exist positive constants $a_*$, $a^*$, $\phi_*$ and $\phi^*$, such that
\begin{equation}\label{eqn.boundaphi}
	a_* \le a(z,\bx) \le a^*, \qquad \phi_* \le \phi(\bx) \le \phi^*,
\end{equation}
for all $\bx \in \Omega$ and $z=z(t)$. In order to simplify the presentation, we define
\[A(z)(\bx):=a^{-1}(z,\bx).\]
Additionally, assume the realistic relation of the diffusion and dispersion coefficients given by $0 <d_m \le d_t\le d_\ell.$

\medskip

\noindent The weak formulation of \eqref{eqn.model}-\eqref{eqn.ic} seeeks $c \in L^2(0,T;Z)\cap C^0(0,T;L^2(\Omega))$, $\bu \in L^2(0,T;\bV)$, and $p \in L^2(0,T;Q)$, such that
\begin{subequations}\label{eqn.weak}
\begin{align}
\cM(\frac{\partial c(t)}{\partial t},z)+(\bu(t)\cdot \nabla c(t),z)+\cD(\bu(t);c(t),z)&=(q^+(\hc-c)(t),z), \fl z \in Z\label{eqn.weakz}\\
\cA(c(t);\bu(t),\bv)+B(\bv,p(t))&=(\bg(c(t)),\bv), \fl \bv \in \bV \label{eqn.weakv}\\
B(\bu(t),q)&=-(G(t),q), \fl q \in Q \label{eqn.weakq}
\end{align}
\end{subequations}
for almost all $t \in J$ and with initial condition $c(0)=c_0$, where
\begin{subequations}
\begin{align}
&\cM(c,z):=(\phi c,z), \quad  \cD(u;c,z):=(D(\bu)\nabla c, \nabla z), \,\;\label{defn.bilinear}\\
& \cA(c;\bu,\bv):=(A(c)\bu,\bv), \quad  B(\bv,q):=-(\divc v,q).\label{defn.bilinear_1}
\end{align}
\end{subequations}
Existence of weak solutions \eqref{eqn.weak} to \eqref{eqn.model}-\eqref{eqn.ic} have been established in \cite{Feng_1995} and \cite{Chen_1999}. For the sake of readability, here and throughout this paper, we write $\bu$ for $\bu(t)$ and for the other functions depending on space and time. The interpretation of whether $\bu$ represents a function of space only or a function of both space and time should be inferred from the surrounding context. 

\medskip

\noindent As in \cite{Veiga_miscibledisplacement_2021}, the following alternative form is used for \eqref{eqn.weakz} as this helps to preserve the properties of the continuous bilinear form after discretisation.
\begin{equation}
\cM(\frac{\partial c}{\partial t},z)+\Theta(\bu, c;z)+\cD(\bu;c,z)=(q^+\hc,z), \fl z \in Z,\label{eqn.weakz_alt}
\end{equation}
where
\[\Theta(\bu,c;z):=\half[(\bu\cdot\nabla c,z)+((q^++q^-)c,z)-(\bu,c\nabla z)].\]

\noindent The kernel is defined as
\begin{equation}\label{defn.kernel}
\cK:=\{\bv \in \bV:B(\bv,q)=0 \fl q \in Q\}.
\end{equation}

\begin{lem}[Properties of the bilinear forms]\label{lem.propertiescontinuous}The following properties hold for the bilinear forms in \eqref{defn.bilinear}-\eqref{defn.bilinear_1}\cite[Section 2.2]{Veiga_miscibledisplacement_2021}:
	\begin{itemize}
		\item[$(a)$] $\cM(c,z) \le \phi^*\|c\|\|z\|$ for all $c,z \in Z$,
		\item[(b)] $\cM(z,z) \ge \phi_*\|z\|^2$ for all $z \in Z$,
		\item[(c)] $\cA(c;\bu,\bv) \le \frac{1}{a_*}\|\bu\|\|\bv\|$ for all $c \in L^\infty(\Omega)$ and $\bu,\bv \in (L^2(\O))^2$,
		\item[(d)] $\cA(c;\bv,\bv) \le \|A(c)\|\|\bu\|_{0,\infty,\O}\|\bv\|$ for all $ c\in L^2(\O)$, $\bu \in (L^\infty(\O))^2$, and $\bv \in (L^2(\O))^2$,
		\item[(e)] $\cA(c;\bv,\bv) \ge \frac{1}{a^*}\|\bv\|^2$ for all $ c \in L^\infty(\O)$ and $\bv \in (L^2(\O))^2$,
		\item[(f)] $\cA(c,\bv,\bv) \ge \frac{1}{a^*}\|\bv\|_\bV^2$ for all $ c \in L^\infty(\O)$ and $\bv \in \cK$,
		\item[(g)] $\cD(\bu;c,z) \le \phi^*[d_m+\|u\|_{0,\infty,\O}(d_\ell+d_t)]\|\nabla c\|\|\nabla z\|$ for all $\bu \in (L^\infty(\O))^2$ and $c,z \in H^1(\O)$,
		\item[(h)] $\cD(\bu;c,z) \le\eta_\cD(1+\|u\|)\|\nabla c\|_{0,\infty,\O}\|\nabla z\|$ for all $\bu \in (L^2(\O))^2$ and $c,z \in H^1(\O)$ with $\nabla c \in L^\infty(\O)$,
		\item[(i)] $(D\bu,\boldsymbol{\mu},\boldsymbol{\mu}) \ge \phi_*(d_m\|\boldsymbol{\mu}\|^2+d_t\||u|^\half\boldsymbol{\mu}\|^2)$ for all $\boldsymbol{\mu} \in (L^2(\O))^2$,
	\end{itemize}
where $\eta_\cD$ is a positive constant depending only on $d_m$, $d_\ell$, and $d_t$.
\end{lem}
\subsection{Main result}
The main result of this paper is briefly presented below. 

\medskip

\noindent Let $(c,\bu,p)$ solves \eqref{eqn.weakz_alt}, \eqref{eqn.weakv}, and \eqref{eqn.weakq}, respectively. Let $0=t_0<t_1<\cdots<t_N=T$ be a given partition of $J=[0,T]$ with time step size $\tau$ and let $h$ be the mesh-size. For $k\ge 0$, let $\bu_h$ be the $H(\divc)$ conforming VEM approximation to $\bu$ of order $k$, $p_h$ be the polynomial approximation to $p$ of order $k$, and $c_h$ be the non-conforming VEM approximation to $c$ of order $k+1$. Then, under mesh assumption and the assumption that the continuous data and solution are sufficiently regular in space and time, there exists a positive constant $\eta$ independent of $h$ and $\tau$ such that, for $n=0,1,\cdots,N$,
\begin{align}
	\|c(t_n)-c_h(t_n)\|+\|\bu(t_n)-\bu_h(t_n)\|+\|p(t_n)-p_h(t_n)\|&\le \eta(\|c_{0,h}-c_0\|+h^{k+1}+\tau),\label{eqn.mainresult}
\end{align}
where $c_{0,h}$ is the interpolant of $c_0$.

\section{The virtual element method}\label{sec:vem}
This section presents the virtual element method for the weak formulation \eqref{eqn.weak}.

\smallskip

\noindent Let $\cT_h$ be a discretisation of $\O$ into polygons $K$. Let $\cE_h$ denote the set of all edges of $\cT_h$, and let $\cE_h^K$ be the set of all edges of $K \in \cT_h$. Let $h_K$ be the diameter of $K$ and mesh-size $h:=\max_{K \in \cT_h} h_K$. Let $ h_e$ be the length of the edge $e$, and $n_K$ be the number of edges of $K$. Assume that there exists a $\rho_0>0$ such that for all $h>0$ and for all $K \in \cT_h$:\\

\noindent (\textbf{D1}) $K$ is star-shaped with respect to a ball of radius $\rho \ge \rho_0h_K$,\\
(\textbf{D2}) $h_e \ge \rho_0h_K$ for all $e \in \cE_h^K$.

\medskip

\noindent Note that these two assumptions imply that the number of edges of each element is uniformly bounded. Additionally,
we will require following quasi-uniformity to prove Lemma~\ref{lem.Pinablainfty}:\\
(\textbf{D3}) for all $h>0$ and for all $K \in \cT_h$, it holds $h_K \ge \rho_1 h$, for some positive uniform constant $\rho_1$.

\subsection{Discrete spaces}
Let $K \in \cT_h$ and let $k \in \N_0$ be a given degree of accuracy. Then the local velocity virtual element space \cite{Veiga_HdivHcurlVEM,MixedVEM_general_2016} is defined by
\begin{align}\label{defn.Vhk}
\bV_h(K):&=\{\bv \in H(\divc;K)\cap H(\rot;K):\, \bv\cdot\bn_{|e} \in \P_k(e) \fl e \in \cE_h^K,\\
&\qquad \quad \divc \bv \in \P_k(K),\, \rot \bv \in \P_{k-1}(K)\}.
\end{align}
Obviously, $(\P_k(K))^2 \subseteq \bV_h(K)$. The degrees of freedom $\{\dof_j^{\bV_h(K)}\}_{j=1}^{\dim \bV_h(K)}$ on $\bV_h(K)$ are
\begin{enumerate}
	\item $\displaystyle \frac{1}{|e|}\int_e \bv \cdot \bn p_k \ds \fl p_k \in \P_k(e) \fl e \in \cE_h^K$\\
	\item $\displaystyle \frac{1}{\sqrt{|K|}}\int_K \divc \bv p_k \dx \fl p_k \in \P_k(K)\setminus \R$\\
	\item $\displaystyle \frac{1}{|K|}\int_K \bv \cdot \bx^\perp p_{k-1} \dx \fl p_{k-1} \in \P_{k-1}(K)$,	
\end{enumerate}
with $\bx^\perp:=(x_2,-x_1)^T$, where we assume the coordinates to be centered at the barycenter of the element.

\medskip

\noindent The local pressure virtual element space \cite{Veiga_HdivHcurlVEM,MixedVEM_general_2016} is
\begin{align}\label{defn.Vhk1}
	Q_h(K):&=\{q \in L^2(K):\, q \in \P_k(K)\}.
\end{align}
Observe that $\P_k(K) \subseteq Q_h(K)$. The degrees of freedom $\{\dof_j^{Q_h(K)}\}_{j=1}^{\dim Q_h(K)}$ on $Q_h(K)$ are
\begin{enumerate}
	\item $\displaystyle \frac{1}{{|K|}}\int_K q p_k \dx \fl p_k \in \P_k(K).$
\end{enumerate}

\medskip

\noindent These two spaces are coupled with the preliminary local concentration spaces \cite{ncVEM_2016}
\begin{align}\label{defn.ZhKtilde}
	\widetilde{Z}_h(K):&=\displaystyle \{z \in H^1(K):\,\frac{\partial z}{\partial n} \in \P_k(e) \fl e \in \cE_h^K, \Delta z \in \P_{k-1}(K)\}
\end{align}
It is clear from \eqref{defn.ZhKtilde} that $\P_{k+1}(K) \subseteq \widetilde{Z}_h(K)$. The degrees of freedom $\{\dof_j^{\widetilde{Z}_h(K)}\}_{j=1}^{\dim \widetilde{Z}_h(K)}$ on $\widetilde{Z}_h(K)$ is defined by
\begin{enumerate}
	\item $\displaystyle \frac{1}{|e|}\int_e zp_k \ds \fl p_k \in \P_k(e) \fl e \in \cE_h^K$\\
	\item $\displaystyle \frac{1}{{|K|}}\int_K zp_{k-1} \dx \fl p_{k-1} \in \P_{k-1}(K).$
\end{enumerate}

\noindent Note that $k=0$ provides the lowest order local VE spaces. Let $\bPi_k^{0,K}:(L^2(K))^2 \to (\P_k(K))^2$ be the $L^2$ projector onto the vector-valued polynomials of degree atmost $k$ in each component. That is, for a given $\bsf \in (L^2(\O))^2$,
\begin{equation}\label{defn.Pi}
(\bPi_k^{0,K}\bsf,\bp_k)=(\bsf,\bp_k) \fl \bp_k \in (\P_k(K))^2.
\end{equation}
This operator is computable for functions in $\bV_h(K)$ only by knowing their values at the degrees of freedom. Also, an integration by parts leads to, for $z_h \in \widetilde{Z}_h(K)$,
\[\int_K \bPi_k^{0,K}\nabla z_h \cdot \bp_k \dx=\int_K \nabla z_h \cdot \bp_k \dx=-\int_K z_h \divc \bp_k \dx +\int_{\partial K}z_h\bp_k\cdot \bn \ds,\]
for all $\bp_k \in (\P_k(K))^2$. The right-hand side is computable using the degrees of freedom of $\widetilde{Z}_h(K)$ and so is the left-hand side.

\smallskip

\noindent In addition to the $L^2$ projector described in \eqref{defn.Pi}, one needs the elliptic projector $\Pi_{k+1}^{\nabla,K}:H^1(K) \to \P_{k+1}(K)$, which is defined as follows:
\begin{subequations}\label{defn.Pielliptic}
\begin{align}
(\nabla \Pi_{k+1}^{\nabla,K}z, \nabla p_{k+1})&=(\nabla z,\nabla p_{k+1}) \fl p_{k+1} \in \P_{k+1}(K) \label{eqn.Pielliptic.a}\\
\frac{1}{|\partial K|}\int_{\partial K} \Pi_{k+1}^{\nabla,K}z&=\frac{1}{|\partial K|}\int_{\partial K} z \mbox{ for } k=0\label{eqn.Pielliptic.k0}\\
\frac{1}{|K|}\int_{ K} \Pi_{k+1}^{\nabla,K}z&=\frac{1}{| K|}\int_{ K} z \mbox{ for } k\ge 1\label{eqn.Pielliptic.k1}
\end{align}
\end{subequations}
for all $z \in H^1(K)$. Note that $\Pi_{k+1}^{\nabla,K}p_{k+1}=p_{k+1}$ for all $p_{k+1} \in \P_{k+1}(K)$. For any $z \in \widetilde{Z}_h(K)$, $\Pi_{k+1}^{\nabla,K}z$ can be computed using integration by parts and the degrees of freedom of $\widetilde{Z}_h(K)$.

\smallskip

\noindent Since the projections in the $L^2$ norm are available only on polynomials of degree $\le k-1$ directly from the degrees of freedom of $\widetilde{Z}_h(K)$, in order to compute the $L^2$ projections on $\P_{k+1}(K)$, we consider a modified virtual element space \cite{Equivalentprojectors_2013} for the concentration. Define
\begin{align}\label{defn.ZhK}
{Z}_h(K):&=\displaystyle \{z \in H^1(K):\,\frac{\partial z}{\partial n} \in \P_k(e) \fl e \in \cE_h^K, \Delta z \in P_{k+1}(K),\\
&\qquad  \int_K \Pi_{k+1}^{\nabla,K}z p_{k+1}\dx= \int_K z p_{k+1}\dx \fl p_{k+1} \in \P_{k+1}(K)\setminus\P_{k-1}(K)\},
\end{align}
where $\P_{k+1}(K)\setminus\P_{k-1}(K)$ is the space of polynomials in $\P_{k+1}(K)$ which are $L^2(K)$ orthogonal to $\P_{k-1}(K)$. It can be shown that the space $Z_h(K)$ has the same degrees of freedom and the same dimension as $\widetilde{Z}_h(K)$ \cite{Equivalentprojectors_2013,ncVEM_Stokes}.

\smallskip
\noindent In the sequel, the notation ``$a\lesssim b$ (resp. $ a \gtrsim b$)'' means that there exists a generic constant $C$ independent of the mesh parameter $h$ and time step size $\tau$ such that $a \le Cb$ (resp. $a \ge C b$). The approximation properties for the projectors are stated next \cite[Lemma 5.1]{VEM_general_2016}, \cite[Lemma 3.1]{Veiga_miscibledisplacement_2021}. Note that the last property for $\ell=1$ can be derived using \eqref{eqn.Pielliptic.a} and an introduction of $\Pi_{k}^{0,K}$. The case $\ell=0$ can be then proved with the help of \Poincare-Fredrich inequalities together with \eqref{eqn.Pielliptic.k0} and \eqref{eqn.Pielliptic.k1}. These estimates together with inverse inequality lead to the case $\ell>1$.

\begin{lem}[Approximation properties]\label{lem.approx}Given $K \in \cT_h$, let $\psi$ and $\bpsi$ be sufficiently smooth scalar and vector-valued functions, respectively. Then, it holds, for all $k \in \N_0$,
	\begin{align}
&(a)\,	\|\psi-\Pi_k^{0,K}\psi\|_{\ell,K} \lesssim h_K^{s-\ell} |\psi|_{s,K}, \quad 0 \le \ell \le s\le k+1\\
&(b)\,	\|\bpsi-\bPi_k^{0,K}\bpsi\|_{\ell,K} \lesssim h_K^{s-\ell} |\bpsi|_{s,K}, \quad 0 \le \ell \le s\le k+1\\
&(c)\,	\|\psi-\Pi_k^{\nabla,K}\psi\|_{\ell,K} \lesssim h_K^{s-\ell} |\psi|_{s,K}, \quad 0 \le \ell \le s\le k+1, \, s\ge 1.
	\end{align}
\end{lem}
\begin{proof}
	 For $(a)-(b)$, see \cite[Lemma 5.1]{VEM_general_2016}. Consider the case $\ell=1$ for $(c)$. The definition of $\Pi_k^{\nabla,K}$ in \eqref{eqn.Pielliptic.a}, an introduction of $\Pi_{k}^{0,K}$, and \Holder inequality show
	 \begin{align}
	 |\psi-\Pi_{k}^{\nabla,K}\psi|_{1,K}^2&=(\nabla(\psi-\Pi_{k}^{\nabla,K}\psi),\nabla(\psi-\Pi_{k}^{\nabla,K}\psi))\\&=(\nabla(\psi-\Pi_{k}^{\nabla,K}\psi),\nabla\psi)\\
	& =(\nabla(\psi-\Pi_{k}^{\nabla,K}\psi),\nabla\psi-\nabla \Pi_{k}^{0,K}\psi)\le |\psi-\Pi_{k}^{\nabla,K}\psi |_{1,K}| \psi-\Pi_{k}^{0,K}\psi|_{1,K}.
	 \end{align}
	 This and $(a)$ lead to the required estimate. For $\ell=0$, we consider the \Poincare-Friedrich inequality\cite{Brenner_PFI} given by, for all $\xi \in H^1(K)$, 
	 \begin{align}
	  h_K^{-1} \|\xi\|& \lesssim h_K^{-1}\left|\int_{\partial K}\xi \ds\right| +|\xi|_{H^1(K)} \label{PFI.a}\\
	  h_K^{-1} \|\xi\|& \lesssim h_K^{-1}\left|\int_{ K}\xi \dx\right| +|\xi|_{H^1(K)}.\label{PFI.b}
	 \end{align}
	 The estimate \eqref{PFI.a} and \eqref{eqn.Pielliptic.k0} (resp. \eqref{PFI.b} and \eqref{eqn.Pielliptic.k1}) together with the property $(c)$ for $\ell=1$ concludes the proof for $\ell=0$. The result for $\ell \ge 2$ follows from an introduction of $\Pi_{k}^{0,K}$, an inverse estimate \cite{DG_DA} for the polynomials (from $H^\ell(K)$ to $H^1(K)$), and the property $(c)$ for $\ell=1$.
\end{proof}

\noindent For every decomposition $\cT_h$ of $\Omega$ into simple polygons $K$, define the global spaces by
\begin{align*}
\bV_h:&=\{\bv \in \bV:\bv_{|K} \in \bV_h(K) \fl K \in \cT_h\}\\
Q_h:&=\{q \in Q:q_{|K} \in Q_h(K) \fl K \in \cT_h\}\\
Z_h:&=\{z \in H^{1,\rm{nc}}(\cT_h;k):z_{|K} \in Z_h(K) \fl K \in \cT_h\},
\end{align*}
where
\[H^{1,\rm{nc}}(\cT_h;k):=\{z \in H^1(\cT_h):\int_e \jump{z}\cdot\bn q\ds=0 \fl q \in P_k(e) \fl e \in \cE_h\}\]
Define, for all $\bu_h \in \bV_h$,
\[\|\bu_h\|_{\bV_h}^2:=\sum_{K \in \cT_h}\|\bu_h\|_{V,K}^2:=\sum_{K \in \cT_h}[\|\bu_h\|^2+\|\divc \bu_h\|^2].\]
Let the $\Pi_k^0$, $\bPi_k^0$, and $\Pi_{k+1}^\nabla$ be the global projectors such that for all $K \in \cT_h$,
\[\Pi^0_{{k}_{|K}}=\Pi_k^{0,K},\,\, \bPi^0_{{k}_{|K}}=\bPi_k^{0,K},\,\, \Pi^\nabla_{{k+1}_{|K}}=\Pi_{k+1}^{\nabla,K}.\]
The sets of global degrees of freedom $\{\dof_j^{\bV_h}\}_{j=1}^{\dim \bV_h}$,$\{\dof_j^{Q_h}\}_{j=1}^{\dim Q_h}$ and $\{\dof_j^{Z_h}\}_{j=1}^{\dim Z_h}$ are achieved by linking together their corresponding local counterparts.

\subsection{Semi-discrete formulation}\label{sec.semidiscrete}
This section deals with the semi-discrete weak formulation of \eqref{eqn.weak} which is continuous in time and discrete in space. Let $\Pi_0^{0,K}:L^2(K)\to \P_0(K)$ and $\bPi_0^{0,K}:(L^2(K))^2\to (\P_0(K))^2$ be the $L^2$ projectors onto the scalar and vector valued functions.

\noindent The semi-discrete variational formulation of \eqref{eqn.weak} seeks $\bu_h \in \bV_h$, $p_h \in Q_h$, and $c_h \in Z_h$ such that, for almost every $t \in J$,
\begin{subequations}\label{eqn.semidiscrete}
\begin{align}
\cM_h(\frac{\partial c_h}{\partial t},z_h)+\Theta_h(\bu_h,c_h;z_h)+\cD_h(\bu_h;c_h,z_h)&=(q^+\hc,z_h), \fl z_h \in Z_h\label{eqn.semidiscreteqz}\\
\cA_h(c_h;\bu_h,\bv_h)+B(\bv_h,p_h)&=(\bg(c_h),\bv_h)_h, \fl \bv_h \in \bV_h \label{eqn.semidiscreteqv}\\
B(\bu_h,q_h)&=-(G,q_h), \fl q_h \in Q_h\label{eqn.semidiscreteqq}
\end{align}
\end{subequations}
with the initial condition 
\[c_h(0)=c_{0,h}:=I_hc_0,\]
where $I_hc_0$ is the interpolant of $c_0$ in $Z_h$. Here,
\begin{enumerate}
	\item The term $\cM_h(\bullet,\bullet)$ in \eqref{eqn.semidiscreteqz} is defined by
\begin{align}
\cM_h\left(\frac{\partial c_h}{\partial t},z_h\right)&:=\sum_{K\in \cT_h} \cM_h^K\left(\frac{\partial c_h}{\partial t},z_h\right), \label{defn.Mh}
\end{align}
where 
\begin{align}
\cM_h^K\left(c_h,z_h\right):=&\int_K \phi (\Pi_{k+1}^{0,K}c_h)(\Pi_{k+1}^{0,K}z_h)\dx+ \nu_\cM^K(\phi) S_\cM^K((I-\Pi_{k+1}^{0,K})c_h,(I-\Pi_{k+1}^{0,K})z_h)
\end{align}
with $S_\cM^K(\bullet,\bullet)$ denotes the stabilization term with property given below in \eqref{eqn.SM} and 
\[\nu_M^K(\phi)=|\Pi_{0}^{0,K}\phi|.\]
	\item The term $\Theta_h(\bullet,\bullet;\bullet)$ in \eqref{eqn.semidiscreteqz} is defined by
\begin{align}
\Theta_h(\bu_h,c_h;z_h)&:=\half[(\bu_h\cdot\nabla c_h,z_h)_h+((q^++q^-)c_h,z_h)_h-(\bu_h,c_h\nabla z_h)_h],\label{eqn.Thetah}
\end{align}
where 
\begin{align*}
(\bu_h\cdot\nabla c_h,z_h)_h&=\sum_{K \in \cT_h}\int_K \bPi_{k}^{0,K}\bu_h\cdot \bPi_{k}^{0,K}(\nabla c_h)\Pi_{k+1}^{0,K}z_h \dx\\
((q^++q^-)c_h,z_h)_h&=\sum_{K \in \cT_h}\int_K(q^++q^-) \Pi_{k+1}^{0,K}c_h\Pi_{k+1}^{0,K}z_h \dx\\
(\bu_h,c_h\nabla z_h)_h&=\sum_{K \in \cT_h}\int_K \bPi_{k}^{0,K}\bu_h \Pi_{k+1}^{0,K} c_h\cdot\bPi_{k}^{0,K}(\nabla z_h) \dx.
\end{align*}
\item The term $\cD_h(\bullet;\bullet,\bullet)$ in \eqref{eqn.semidiscreteqz} is defined by
\begin{align}
\cD_h\left(\bu_h,c_h;z_h\right)&:=\sum_{K\in \cT_h} \cD_h^K\left(\bu_h,c_h;z_h\right),
\end{align}
where 
\begin{align}
\cD_h^K\left(\bu_h,c_h;z_h\right):=&\int_K D(\bPi_{k}^{0,K}\bu_h) \bPi_{k}^{0,K}(\nabla c_h)\cdot\bPi_{k}^{0,K}(\nabla z_h)\dx\\
&\qquad + \nu_\cD^K(\bu_h)S_\cD^K((I-\Pi_{k+1}^{\nabla,K})c_h,(I-\Pi_{k+1}^{\nabla,K})z_h)
\end{align}
with $S_\cD^K(\bullet,\bullet)$ denotes the stabilization term with property given below in \eqref{eqn.SD} and 
\[\nu_D^K(\bu_h)=\nu_\cM^K(\phi)(d_m+d_t|\bPi_{0}^{0,K}\bu_h|).\]
\item The term $(q^+\hc,\bullet)$  in \eqref{eqn.semidiscreteqz} is defined by
\begin{equation}
(q^+\hc,z_h)_h:=\sum_{K\in \T_h}\int_K q^+\hc\Pi_{k+1}^{0,K}z_h \dx.
\end{equation}
\item The term $\cA_h(\bullet;\bullet,\bullet)$ in \eqref{eqn.semidiscreteqv} is defined by
\begin{align}
\cA_h\left(c_h;\bu_h,\bv_h\right)&:=\sum_{K\in \cT_h} \cA_h^K\left(c_h;,\bu_h,\bv_h\right),
\end{align}
where 
\begin{align}
\cA_h^K\left(c_h;\bu_h,\bv_h\right):=&\int_K A(\Pi_{k+1}^{0,K}c_h) \bPi_{k}^{0,K}(\bu_h)\cdot\bPi_{k}^{0,K}(\bv_h)\dx\\
&\qquad + \nu_\cA^K(c_h) S_\cA^K((I-\bPi_{k}^{0,K})\bu_h,(I-\bPi_{k}^{0,K})\bv_h)
\end{align}
with $S_\cA^K(\bullet,\bullet)$ denotes the stabilization term with property given below in \eqref{eqn.SA} and 
\[\nu_A^K(c_h)=|A(\Pi_{0}^{0,K}c_h)|.\]
\item The term $(\bg(\bullet),\bullet)_h$ in \eqref{eqn.semidiscreteqv} is defined by
\begin{equation}
(\bg(c_h),\bv_h)_h:=\sum_{K \in \T_h}\int_K \bg(\Pi_{k+1}^{0,K}c_h)\cdot \bPi_{k}^{0,K}\bv_h\dx.
\end{equation}
\end{enumerate}
The stabilisation terms $S_\cM:Z_h\times Z_h\to \R$, $S_\cD:Z_h\times Z_h \to \R$, and $S_\cA:\bV_h \times \bV_h \to \R$ are symmetric and positive definite bilinear forms with the property that for all $K \in \cT_h$, there exists positive constants $M_0^\cM,M_1^\cM,M_0^\cD,M_1^\cD,M_0^\cA,$ and $M_1^\cA$ independent of $h$ and $K$ such that
\begin{subequations}
\begin{align}
&M_0^\cM \|z_h\|_{0,K}^2 \le S_\cM^K(z_h,z_h) \le M_1^\cM \|z_h\|_{0,K}^2 \quad \fl z_h \in Z_h \cap \ker(\Pi_{k+1}^{0,K})\label{eqn.SM}\\
&M_0^\cD \|\nabla z_h\|_{0,K}^2 \le S_\cD^K(z_h,z_h) \le M_1^\cD \|\nabla z_h\|_{0,K}^2 \quad \fl z_h \in Z_h \cap \ker(\Pi_{k+1}^{\nabla,K})\label{eqn.SD}\\
&M_0^\cA \|\bv_h\|_{0,K}^2 \le S_\cA^K(\bv_h,\bv_h) \le M_1^\cA \|\bv_h\|_{0,K}^2 \quad \fl \bv_h \in \bV_h \cap \ker(\bPi_{k}^{0,K}).\label{eqn.SA}
\end{align}
\end{subequations}
The above properties prove the continuity of $\cM_h^K(\bullet,\bullet)$, $\cD_h^K(\bullet,\bullet;\bullet)$, and $\cA_h^K(\bullet,\bullet;\bullet)$ respectively. Under the mesh assumption (\textbf{D1})-(\textbf{D2}), a simplier choice of these stabilisation terms are given by
\begin{align*}
S_\cM^K(c_h,z_h)&=|K|\sum_{j=1}^{\dim Z_h(K)}\dof_j^{Z_h(K)}(c_h) \dof_j^{Z_h(K)} (z_h)\\
S_\cD^K(c_h,z_h)&=\sum_{j=1}^{\dim Z_h(K)}\dof_j^{Z_h(K)}(c_h) \dof_j^{Z_h(K)} (z_h)\\
S_\cA^K(\bu_h,\bv_h)&=|K|\sum_{j=1}^{\dim \bV_h(K)}\dof_j^{\bV_h(K)}(\bu_h) \dof_j^{\bV_h(K)} (\bv_h).
\end{align*}

\noindent The lemma stated below shows the continuity and coercivity properties of the discrete bilinear form in \eqref{eqn.semidiscrete}.The proof follows analogous to \cite[Lemma 3.2]{Veiga_miscibledisplacement_2021} with \cite[Lemma 3.1.c]{Veiga_miscibledisplacement_2021} replaced with Lemma~\ref{lem.approx}.c and hence is skipped.
\begin{lem}[Properties of the discrete bilinear forms]\label{lem.propertiesdiscrete} The following properties hold for the discrete bilinear forms in \eqref{eqn.semidiscrete}
	\begin{itemize}
		\item[$(a)$] $\cM_h(c_h,z_h) \lesssim \|c_h\|\|z_h\|$ for all $c_h,z_h\in Z_h$,
		\item[(b)] $\cM_h(z_h,z_h) \gtrsim\|z_h\|^2$ for all $z_h \in Z_h$,
			\item[(c)] $\cD_h(\bu_h;c_h,z_h) \lesssim | c_h|_{1,\cT_h}|z_h|_{1,\cT_h}$ for all $\bu_h \in \bV_h$ and $c_h,z_h \in Z_h$,
		\item[(d)] $\cD_h(\bu_h;z_h,z_h) \gtrsim |z_h|_{1,\cT_h}^2$ for all $\bu_h \in \bV_h$ and $z_h \in Z_h$,
	\item[(e)] $\cA_h(c_h;\bu_h,\bv_h) \lesssim \|\bu_h\|\|\bv_h\|$ for all $c_h \in Z_h$ and $\bu_h,\bv_h \in \bV_h$,
		\item[(f)] $\cA_h(c_h;\bv_h,\bv_h) \gtrsim \|\bv_h\|^2$ for all $c_h \in Z_h$ and $\bv_h \in \bV_h$.
	\end{itemize}
Thus, $\cA_h(c_h,\bullet,\bullet)$ is coercive on the kernel
\begin{equation}\label{defn.kerneldiscrete}
\cK_h:=\{\bv_h \in \bV_h:B(\bv_h,q_h)=0 \fl q_h \in Q_h\} \subset \cK.
\end{equation}
with respect to $\|\bullet\|_{\bV_h}$, where $\cK$ is given in \eqref{defn.kernel}.
\end{lem}
\noindent Well-posedness of \eqref{eqn.semidiscrete} can be established using the tools of \cite{ncVEM_parabolic} which deals with the nonconforming virtual element method for parabolic problems and \cite{mixedVEM_Brezzi} for the mixed virtual element method together with \ref{lem.propertiesdiscrete}. More precisely, as in \cite{miscibledisplacement_1983}, for a given $c_h(t) \in L^\infty(\O)$, there exists a unique solution $(\bu_h(c_h),p_h(c_h))$ for \eqref{eqn.semidiscreteqv}-\eqref{eqn.semidiscreteqq}. A subtitution of this solution in \eqref{eqn.semidiscreteqz} leads to a system of non-linear differential equation in $c_h$. Picards theorem and an a priori bound for $c_h$ then show the existence and uniqueness of the discrete concentration $c_h(t)$ for all $t \in J$.
\subsection{Fully discrete formulation}
\noindent A semi-discrete formulation of \eqref{eqn.weak} is presented in Section~\ref{sec.semidiscrete}. This section deals with the fully discrete formulation which is discrete in both space and time. The temporal discretization is achieved through the utilization of a backward Euler method.

\smallskip

\noindent Let $0=t_0<t_1<\cdots<t_N=T$ be a given partition of $J=[0,T]$ with time step size $\tau$. That is, $t_n=n\tau$, $n=0,1,\cdots,N$. For a generic function $f(t)$, define $f^n:=f(t_n)$, $n=0,1,\cdots,N$. Also, define
\[\bu^n:=\bu(t_n),\, p^n:=p(t_n),\, c^n=c(t_n)\]
and
\[\bu_h^n:=\bu_h(t_n),\, p_h^n:=p_h(t_n),\, c_h^n=c_h(t_n).\]
At $t_n$ for $n=0,1,\cdots,N$ with $c_h^0=c_{0,h}$, the fully discrete formulation corresponding to the velocity-pressure equation seeks $(\bu_h^n,p_h^n) \in \bV_h \times Q_h$ such that
\begin{subequations}\label{eqn.fullydiscrete_uhph}
	\begin{align}
	\cA_h(c_{h}^n;\bu_h^n,\bv_h)+B(\bv_h,p_h^n)&=(\bg(c_h^n),\bv_h)_h, \fl \bv_h \in \bV_h \label{eqn.fullydiscreteqv}\\
	B(\bu_h^n,q_h)&=-(G^n,q_h), \fl q_h \in Q_h.\label{eqn.fullydiscreteqq}
	\end{align}
\end{subequations}
Once $(\bu_h^n,p_h^n)$ is solved, the approximation to concentration at time $t=t_{n+1}$ can obtained with the help of  $(\bu_h^n,p_h^n)$ and the Euler scheme for the time derivative $\frac{\partial c_h^{n+1}}{\partial t}$ given by
\[\frac{\partial c_h^{n+1}}{\partial t}_{|t=t_{n+1}}\approx \frac{c_h^{n+1}-c_h^n}{\tau}.\]
 The fully discrete formulation corresponding to the concentration equation seeks $c_h^{n+1}\in Z_h$ such that
\begin{align}
	\cM_h(\frac{ c_h^{n+1}-c_h^n}{\tau},z_h)+\Theta_h(\bu_h^n,c_h^{n+1};,z_h)+\cD_h(\bu_h^n;c_h^{n+1},z_h)&=(q^{+(n+1)}\hc^{(n+1)},z_h)_h, \fl z_h \in Z_h.\label{eqn.fullydiscreteq_zh}
	\end{align}
	 Note that \eqref{eqn.fullydiscrete_uhph} and \eqref{eqn.fullydiscreteq_zh} are decoupled from each other and hence represent system of linear equations eventhough the original problem is a nonlinear coupled system problem for concentration, pressure, and velocity.
\smallskip

\noindent	It can be established analogous to that of \cite[Lemma 3.3]{Veiga_miscibledisplacement_2021} that under the assumption $G^n,q^{+n},p_h^n,c_h^n \in L^\infty(\O)$, $\bg(c_h^n) \in (L^2(\O))^2$, and $\bu_h^n \in (L^\infty(\O))^2$ for all $n=0,1,\cdots,N$, for a given $\tau>0$, \eqref{eqn.fullydiscrete_uhph}-\eqref{eqn.fullydiscreteq_zh} is uniquely solvable.
 
 \section{Error estimates}\label{sec:error}
  This section is devoted to the convergence analysis of the scheme.
  
  \begin{lem}[Auxiliary result]\cite[Lemma 4.1]{Veiga_miscibledisplacement_2021}\label{lem.aux}
  	Let $r,s,t \in \N_0$. Let $\Pi_r^0$ and $\bPi_s^0$ denote the elementwise $L^2$ projectors onto scalar and vector valued polynomials of degree at most $r$ and $s$ respectively. Given a scalar function $\sigma \in H^{m_r}(\cT_h)$, $0 \le m_r \le r+1$, let $\kappa(\sigma)$ be a tensor valued piecewise Lipschitz continuous with respect to $\sigma$. Further, let $\widehat{\sigma} \in L^2(\O)$ and let $\bchi$ and $\bpsi$ be vector valued functions. Assume that $\kappa(\sigma) \in (L^\infty(\O))^{2 \times 2}$, $\bchi \in (H^{m_s}(\cT_h)\cap L^\infty(\O))^2$, $\bpsi \in (L^2(\O))^2$ and $\kappa(\sigma)\bchi \in (H^{m_t}(\cT_h))^2$, for some $0 \le m_s\le s+1$ and $0 \le m_t \le t+1.$ Then, for any $K \in \cT_h$,
  	\begin{align*}
  	(\kappa(\sigma)\bchi,\bpsi)_{0,K}-(\kappa(\Pi_r^{0,K}\widehat{\sigma})\bPi_s^{0,K}\bchi,\bPi_t^{0,K}\bpsi)_{0,K}
  &	 \le \eta \big[h^{m_t}|\kappa(\sigma)\bchi|_{{m_t},K}+h^{m_s}|\bchi|_{{m_s},K}\|\kappa(\sigma)\|_{0,\infty,K}\\
  	 &\quad+(h^{m_r}|\sigma|_{{m_r},K} +\|\sigma-\widehat{\sigma}\|_{0,K})\|\bchi\|_{0,\infty,K}\big]\|\bpsi\|_{0,K}.
  	\end{align*}
  	Consequently,
  		\begin{align*}
  	(\kappa(\sigma)\bchi,\bpsi)-(\kappa(\Pi_r^0\widehat{\sigma})\bPi_s^0\bchi,\bPi_t^0\bpsi)
  	&\le \eta \big[h^{m_t}|\kappa(\sigma)\bchi|_{{m_t},\T_h}+h^{m_s}|\bchi|_{{m_s},\T_h}\|\kappa(\sigma)\|_{0,\infty,\O}\\
  	&\quad+(h^{m_r}|\sigma|_{{m_r},\T_h}+\|\sigma-\widehat{\sigma}\|)\|\bchi\|_{0,\infty,\O}\big]\|\bpsi\|.
  	\end{align*}
  \end{lem}
\noindent Consider the mixed problem
\begin{subequations}\label{eqn.uhph.error}
  \begin{align}
  \cA_h(c_{h}^n;\bu_h^n,\bv_h)+B(\bv_h,p_h^n)&=(\bg(c_h^n),\bv_h)_h, \fl \bv_h \in \bV_h \label{eqn.errorv}\\
  B(\bu_h^n,q_h)&=-(G^n,q_h), \fl q_h \in Q_h,\label{eqn.errorq}
  \end{align}
\end{subequations}
where $c_h^n \in Z_h$ is the numerical solution of the concentration equation \eqref{eqn.fullydiscreteq_zh} for $n=1,\cdots,N$ and $c_h^0=c_{0,h}$.

\smallskip

\noindent The error bounds for velocity and pressure are stated below which follows from \cite[Theorem 1]{Veiga_miscibledisplacement_2021} and is therefore skipped.
\begin{thm}[Error for velocity and pressure]\label{thm.up}
	Given $c_h^n \in Z_h$, let $(\bu_h^n,p_h^n) \in \bV_h \times Q_h$ be the solution to \eqref{eqn.uhph.error}. Assume that $\bu^n \in (H^{k+1}(\cT_h))^2,$ $p^n \in H^{k+1}(\cT_h)$, and $c^n \in H^{k+1}(\cT_h)$ where $(\bu^n,p^n,c^n)$ solves \eqref{eqn.weak} at time $t_n$ and $k \in \N_0$. Furthermore, assume that $\bg(c)$ and $A(c)$ are piecewise Lipschitz continuous functions with respect to $c \in L^2(\O)$, and $\bg(c^n)$, $A(c^n)\bu^n \in (H^{k+1}(\cT_h))^2$. Then
	\begin{align*}
	\|\bu^n-\bu_h^n\|& \lesssim h^{k+1}+\|c^n-c_h^n\|\\
	\|p^n-p_h^n\|& \lesssim h^{k+1}+\|c^n-c_h^n\|.
	\end{align*}
\end{thm}
\subsection{Error analysis for concentration}
For a fixed $\bu(t) \in \bV$ and $t \in J$, define the projector $\P_c:Z \cap H^2(\O) \to Z_h$ \cite{ncVEM_2016,ncVEM_parabolic} by
\begin{equation}\label{eqn.projectionP}
\Gamma_{c,h}(\bu(t);\P_cc,z_h)=\Gamma_{c,\pw}(\bu(t);c,z_h)-\cN_h(\bu(t);c,z_h) \fl z_h \in Z_h,
\end{equation}
where
\begin{align}
\Gamma_{c,h}(\bu;c_h,z_h)&:=\cD_h^\bu(c_h,z_h)+\Theta_h^\bu(c_h,z_h)+(c_h,z_h)_h\\
\Gamma_{c,\pw}(\bu;c,z_h)&:=\cD_{\pw}^\bu(c,z_h)+\Theta_{\pw}^\bu(c,z_h)+(c,z_h)\\
\cN_h(\bu;c,z_h)&:=\sum_{e \in \cE_h}\int_e\left(D(\bu)\nabla c \cdot \bn_e-\frac{c\bu \cdot n_e}{2} \right)\jump{z_h}\ds
\end{align}
with
\begin{align}
\cD_h^\bu\left(c_h,z_h\right)&:=\sum_{K \in \cT_h}\int_K D(\bu) \bPi_{k}^{0,K}(\nabla c_h)\cdot\bPi_{k}^{0,K}(\nabla z_h)\dx\\
&\qquad + \nu_\cD^K(\bu)S_\cD^K((I-\Pi_{k+1}^{\nabla,K})c_h,(I-\Pi_{k+1}^{\nabla,K})z_h)\\
\Theta_h^\bu(c_h,z_h)&:=\half\sum_{K\in \cT_h}[(\bu\cdot \bPi_{k}^{0,K}(\nabla c_h),\Pi_{k+1}^{0,K}z_h)_{0,K}-(\bu,\Pi_{k+1}^{0,K} c_h\cdot\bPi_{k}^{0,K}(\nabla z_h))_{0,K}]\\
&\qquad +\half
((q^++q^-)c_h,z_h)_h\label{eqn.Thetahu}\\
 \cD_{\pw}^\bu(\bullet,\bullet)&:=\sum_{K \in \cT_h}\cD^K(\bu;\bullet,\bullet),\qquad \Theta_{\pw}^\bu(\bullet,\bullet):=\sum_{K \in \cT_h}\Theta^K(\bu,\bullet;\bullet)
\end{align}
and $\jump{z_h}$ denotes the jump of $z_h$ across the edge $e$. Here, $\nu_D^K(\bu)=\nu_\cM^K(\phi)(d_m+d_t|\bPi_{0}^{0,K}\bu|)$ and $\cD^K(\bu,\bullet,\bullet)$ and $\Theta^K(\bu,\bullet,\bullet)$ denote the piecewise (elementwise) contribution of $\cD(\bullet,\bullet,\bullet)$ and $\Theta(\bullet,\bullet,\bullet)$, respectively.
\begin{rem}
The operator $\P_c$ in \cite[(5.3)]{Veiga_miscibledisplacement_2021} is defined with a slight variation. In their approach, $\P_c$ involves the $L^2$ projection of the velocity field $\bu$ for its definition. However, our choice is to consider $\bu$ without any projection, given its fixed nature.
\end{rem}
\begin{lem}\cite[Lemma 4.1]{ncVEM_2016}\label{lem.cN}
	Let $k\ge 0$, $c \in H^{k+2}(\O)$, $c\bu \in (H^{k+1}(\cT_h))^2$ and $D(\bu)\nabla c \in (H^{k+1}(\cT_h))^2$. Then,  for all $z_h \in Z_h$,
	 $$\cN_h(\bu;c,z_h) \lesssim h^{k+1}(\|D(\bu)\nabla c\|_{k+1}|z_h|_{1,\cT_h}+\|c\bu\|_{k+1}\|z_h\|).$$
\end{lem}
\begin{lem}\label{lem.projectionPc}
	The projector operator $\P_c$ in \eqref{eqn.projectionP} is well-defined under the assumption that $\bu, q^+$ and $q^-$ are bounded in $L^\infty(\O)$ for all $t \in J$.
\end{lem}
\begin{proof}
To apply Lax-Milgram Lemma, consider the left-hand side of \eqref{eqn.projectionP}. The continuity of $\cD_h^\bu(\bullet,\bullet)$ is analogue to the one in Lemma~\ref{lem.propertiesdiscrete}.c. The definition of $\Theta_h^\bu(\bullet,\bullet)$, a generalised \Holder inequality, inverse estimate $\|\bPi_k^{0,K}\bu\|_{0,\infty,K} \lesssim \|\bu\|_{0,\infty,K}$ from \cite[(41)]{Veiga_miscibledisplacement_2021} and the stability of the $L^2$ projector lead to the continuty of $\Theta_h^\bu(\bullet,\bullet)$. The \Holder inequality and the stability of the $L^2$ projector provide the continuity of $(\bullet,\bullet)_h$. A combination of these estimates shows the continuity of $\Gamma_{c,h}(\bu;\bullet,\bullet)$. For $z_h \in Z_h$, since $((q^{+}+q^{-})z_h,z_h)_h \ge 0$,
\begin{align}
\Theta_h^\bu(z_h,z_h)+(z_h,z_h)_h&=\half[(\bu\cdot\nabla z_h,z_h)_h+((q^++q^-+2)z_h,z_h)_h-(\bu,z_h\nabla z_h)_h]\\&=\half[((q^++q^-+2)z_h,z_h)_h] \ge \|\Pi_{k+1}^{0}z_h\|^2,
\end{align}
where $\Pi_{k+1|K}^0=\Pi_{k+1}^{0,K}$ for any $K \in \cT_h$. This and the coercivity of $\cD_h^\bu(z_h,z_h)$ from Lemma~\ref{lem.propertiesdiscrete}.d (with $\bu_h$ replaced by $\bu$) read
\begin{align}
\Gamma_{c,h}(\bu;z_h,z_h)\gtrsim \|\Pi_{k+1}^{0}z_h\|^2+|z_h|_{1,\cT_h}^2 \ge \|\Pi_{0}^{0}z_h\|^2+|z_h|_{1,\cT_h}^2
\end{align}
with an application of Pythagorus theorem for $\Pi_{k+1}^{0}z_h-\Pi_0^0 z_h$ and $\Pi_0^0z_h$ in the last step. The \Poincare-Friedrich inequality for piecewise $H^1$ function \eqref{PFI.b} provides, for $K \in \cT_h$,
\begin{align}
 \|\Pi_{0}^{0}z_h\|_{0,K}^2+|z_h|_{1,K}^2&=\|\bar{z_h}\|_{0,K}^2+|z_h|_{1,K}^2 \gtrsim \|z_h\|_{0,K}^2+|z_h|_{1,K}^2,
\end{align}
where $\bar{z_h}$ denotes the average of $z_h$ over $K$. A combination of the above two displayed estimates shows the coercivity of $\Gamma_{c,h}(\bu;z_h,z_h)$. 

\smallskip

\noindent  For $\bu \in L^\infty(\O)$, $c \in Z$, and $z_h \in Z_h$,
the generalised \Holder inequality and the definition of $D(\bu)$ in \eqref{defn.D} (also see Lemma \ref{lem.propertiescontinuous}.g) imply
\begin{align}
	\cD_{\pw}^\bu(c,z_h)
	&\le \sum_{K \in \cT_h}\|D(\bu)\|_{0,\infty,K}\|\nabla c\|_{0,K}\|\nabla z_h\|_{0,K}\lesssim \|\bu\|_{0,\infty,\O}\|c\|_{1,\O}\|z_h\|_{1,\cT_h}.\label{eqn.Dpw}
\end{align}
The remaining two terms on the definition of $\Gamma_{c,\pw}(\bu;\bullet,\bullet)$ is estimated as
\begin{align}
	\Theta_{\pw}^\bu (c,z_h)&+(c,z_h)=\sum_{K \in \cT_h}\half[(\bu\cdot\nabla c,z_h)_{0,K}-(\bu,c\nabla z_h)_{0,K}]+\half((q^++q^-+2)c,z_h)\\&
	\le \half(\|\bu\|_{0,\infty,\O}(|c|_{1,\cT_h}\|z_h\|+\|c\||z_h|_{1,\cT_h})+\|q^++q^-+2\|_{0,\infty,\O}\|c\|\|z_h\|)\label{eqn.thethapw}
\end{align}
with a generalised \Holder inequality in the last step. 
The definition of $H^1(\cT_h;k)$ and $L^2(e)$ projection, Cauhy-Schwarz inequality, and approximation properties $\|D(\bu)\nabla c  -\P_k^e(D(\bu)\nabla c)\|_{0,e}\lesssim h^{-1/2}\|D(\bu)\nabla c\|$ and $\|\jump{z_h}-\P_0^e(\jump{z_h})\|_{0,e}\lesssim h^{1/2}|z_h|_{1,\cT_h}$ prove
\begin{align*}
	\cN_h(\bu;c,z_h)&=\sum_{e \in \cE_h}\int_e\left(D(\bu)\nabla c  -\P_k^e(D(\bu)\nabla c)\right)\cdot \bn_e\jump{z_h}\ds\\
	&\qquad -\half\sum_{e \in \cE_h}\int_e\left(\nabla c\bu-\P_k^e(c \bu)\right)\cdot \bn_e\jump{z_h}\ds\\
	&=\sum_{e \in \cE_h}\int_e\left(D(\bu)\nabla c  -\P_k^e(D(\bu)\nabla c)\right)\cdot \bn_e(\jump{z_h}-\P_0^e(\jump{z_h}))\ds\\
	&\qquad -\half\sum_{e \in \cE_h}\int_e\left(\nabla c\bu-\P_k^e(c \bu)\right)\cdot \bn_e(\jump{z_h}-\P_0^e(\jump{z_h}))\ds\\
	& \lesssim \|\bu\|_{0,\infty,\O}|c|_{1,\cT_h}\|z_h\|_{1,\cT_h}.
\end{align*}
\noindent This, \eqref{eqn.thethapw}, and \eqref{eqn.Dpw} show that $\Gamma_{c,\pw}(\bu;c,\bullet)$ is a continuous functional with respect to $\|z_h\|_{1,\cT_h}$. The result then follows from the Lax-Milgram lemma.
\end{proof}
\noindent Given $c \in H^{k+2}(\cT_h)$, there exists an interpolant $c_I \in Z_h$ and a piecewise $\P_{k+1}$ polynomial $c_{\pi}$ such that \cite{ncVEM_2016,ncVEM_parabolic}
\begin{align}
&\|c-c_I\| \lesssim h^{k+2}\|c\|_{k+2,\cT_h}, \quad \|c-c_I\|_{1,\cT_h} \lesssim h^{k+1}\|c\|_{k+2,\cT_h}\label{eqn.cI}\\
&\|c-c_{\pi}\| \lesssim h^{k+2}\|c\|_{k+2,\cT_h}, \quad \|c-c_{\pi}\|_{1,\cT_h} \lesssim h^{k+1}\|c\|_{k+2,\cT_h}.\label{eqn.cpi}
\end{align}
Also, note that 
\begin{equation}\label{eqn.cpiproperties}
\Pi_{k+1}^{0}c_\pi=c_\pi, \quad \Pi_{k+1}^{\nabla}c_\pi=c_\pi.
\end{equation} 
\begin{lem}\label{lem.cPcerror}
	Assume that $\bu \in (H^{k+1}(\cT_h)\cap L^\infty(\O))^2,$ $c \in H^{k+2}(\cT_h) \cap W^{1,\infty}(\cT_h)$, $q^+,q^- \in L^\infty(\O)$, $(q^++q^-)c \in (H^{k+1}(\cT_h))^2$, $\bu \cdot \nabla c \in H^{k+1}(\cT_h)$, and $D(\bu)\nabla c \in (H^{k+1}(\cT_h))^2$ for all $t \in J$. Then 
	\begin{align*}
(a)\,	\|c-\P_c c\|_{1,\cT_h} \lesssim h^{k+1}, \qquad & (b)\,\|c-\P_c c\| \lesssim h^{k+2},
	\end{align*}
	where $\P_c$ is defined in \eqref{eqn.projectionP}.
\end{lem}
\noindent{\it Proof of $(a)$.} The triangle inequality and \eqref{eqn.cI} lead to
	\begin{equation}\label{eqn.tri}
	\|c-\P_c c\|_{1,\cT_h} \le \|c-c_I\|_{1,\cT_h}+\|c_I-\P_c c\|_{1,\cT_h} \lesssim h^{k+1}+\|c_I-\P_c c\|_{1,\cT_h}.
	\end{equation}
	For a fixed time $t$, the coercivity of $\Gamma_{c,h}(\bu;\bullet,\bullet)$ with respect to $\|\bullet\|_{1,\cT_h}$ from the proof of Lemma \ref{lem.projectionPc} and \eqref{eqn.projectionP} show, for $z_h=\P_c c-c_I \in Z_h$,
	\begin{align}
\|z_h\|_{1,\cT_h}^2& \lesssim \Gamma_{c,h}(\bu;z_h,z_h)=\Gamma_{c,h}(\bu;\P_c c,z_h)-\Gamma_{c,h}(\bu;c_I,z_h)\nonumber\\
&=\Gamma_{c,\pw}(\bu;c,z_h)-\Gamma_{c,h}(\bu;c_I,z_h)-\cN_h(\bu;c,z_h) \nonumber\\
&=(\Gamma_{c,\pw}(\bu;c,z_h)-\Gamma_{c,h}(\bu;c_{\pi},z_h))+\Gamma_{c,h}(\bu;c_{\pi}-c_I,z_h)-\cN_h(\bu;c,z_h)\nonumber\\
&=(\cD_{\pw}^\bu(c,z_h)-\cD_h^\bu(c_{\pi},z_h))+(\Theta_{\pw} ^\bu(c,z_h)-\Theta_h^\bu(c_{\pi},z_h))\nonumber\\
&\quad + ((c,z_h)-(c_{\pi},z_h)_h)+\Gamma_{c,h}(\bu;c_{\pi}-c_I,z_h)-\cN_h(\bu;c,z_h)=: A_1+\cdots+A_5 \label{eqn.a1a2a3a4a5}
	\end{align}
	with the definition of $\Gamma_{c,\pw}(\bu;\bullet,\bullet)$ and $\Gamma_{c,h}(\bu;\bullet,\bullet)$, and $\Pi_{k+1}^{\nabla,K}c_\pi=c_\pi$ in the second last step.
An introduction of $\bPi_k^{0,K}(\nabla c)$, Lemma \ref{lem.aux} with the substitutions $\sigma=\bu$, $\widehat{\sigma}=0$, $\bchi=\nabla c$, and $\bpsi=\nabla z_h$, generalised \Holder inequality, \eqref{eqn.cpi}, and the approximation property $\|\bu-\Pi_k^{0,K}\bu\| \lesssim h^{k+1}$ leads to 
\begin{align}
A_{1}&=\sum_{K \in \cT_h} ((D(\bu)\nabla c,\nabla z_h)_{0,K}-(D(\bu)\bPi_k^{0,K}(\nabla c_{\pi}),\bPi_k^{0,K}(\nabla z_h))_{0,K})\\
&=\sum_{K \in \cT_h} ((D(\bu)\nabla c,\nabla z_h)_{0,K}-(D(\bPi_k^{0,K}\bu)\bPi_k^{0,K}(\nabla c),\bPi_k^{0,K}(\nabla z_h))_{0,K})\\
&\qquad +\sum_{K \in \cT_h} (D(\bPi_k^{0,K}\bu)\bPi_k^{0,K}(\nabla (c-c_{\pi})),\bPi_k^{0,K}(\nabla z_h))_{0,K}\\
&\qquad +\sum_{K \in \cT_h} (D(\bPi_k^{0,K}\bu-\bu)\bPi_k^{0,K}(\nabla c_{\pi}),\bPi_k^{0,K}(\nabla z_h))_{0,K}\lesssim h^{k+1}|z_h|_{1,\cT_h}.\label{eqn.a1}
\end{align}
A simple manipulation yields
\begin{align}
2A_{2}
&=\sum_{K \in \cT_h}[((\bu\cdot\nabla c,z_h)_{0,K}-(\bu\cdot\nabla c_\pi,z_h)_h)+(((q^++q^-)c,z_h)_{0,K}-((q^++q^-)c_\pi,z_h)_h)\\
&\qquad -((\bu,c\nabla z_h)_{0,K}-(\bu,c_\pi\nabla z_h)_h)]\\
&= \sum_{K \in \cT_h}[((\bu\cdot\nabla c,z_h)_{0,K}-(\bu\cdot\nabla c,z_h)_h)+(((q^++q^-)c,z_h)_{0,K}-((q^++q^-)c,z_h)_h)\\
&\qquad -((\bu,c\nabla z_h)_{0,K}
 -(\bu,c\nabla z_h)_h)+(\bu\cdot\nabla (c-c_\pi),z_h)_h)\\
 &\qquad +((q^++q^-)(c-c_\pi),z_h)_h-(\bu,(c-c_\pi)\nabla z_h)_h]:=A_{2,1}+\cdots +A_{2,6}.\label{eqn.a2}
\end{align}
Lemma~\ref{lem.aux}, the generalised \Holder inequality, and the continuity property of $L^2$ projectors $\Pi_{k+1}^{0,K}$ and $\bPi_k^{0,K}$ provide
 \begin{equation}
A_{2,1}+A_{2,2}+A_{2,3} \lesssim h^{k+1}\|z_h\|_{1,\cT_h}. \label{eqn.a21a22a23}
\end{equation}The definition of $(\bullet,\bullet)_h$, generalised \Holder inequality, the stability of $L^2$ projectors, and \eqref{eqn.cpi} read
\begin{equation}
A_{2,4}+A_{2,5}+A_{2,6} \lesssim h^{k+1}\|z_h\|_{1,\cT_h}. \label{eqn.a24a25a26}
\end{equation}
A combination of \eqref{eqn.a21a22a23}-\eqref{eqn.a24a25a26} in \eqref{eqn.a2} shows 
\begin{equation}
A_2 \lesssim h^{k+1}\|z_h\|_{1,\cT_h}.
\end{equation}
Since $\Pi_{k+1}^{0}c_\pi=c_\pi$, $(c_{\pi},z_h)_h=(c_{\pi},z_h)$. This, \Holder inequality, and \eqref{eqn.cpi} result in
\begin{align}
A_{3}&=((c-c_{\pi}),z_h) \lesssim h^{k+1}\|z_h\|_{1,\cT_h}.\label{eqn.a3}
\end{align}
The continuity of $\Gamma_{c,h}(\bu;\bullet,\bullet)$ from Lemma \ref{lem.projectionPc}, a triangle inequality with $c$ and the approximation properties in \eqref{eqn.cI} and \eqref{eqn.cpi} provide
\begin{align}
A_4&\lesssim \|\bu\|_{0,\infty,\O}\|c_{\pi}-c_I\|_{1,\cT_h}\|z_h\|_{1,\cT_h}\lesssim h^{k+1}\|z_h\|_{1,\cT_h}.\label{eqn.a4}
\end{align}
A substitution of the estimates for $A_1$, $A_2$, $A_3$, and $A_4$ in \eqref{eqn.a1a2a3a4a5} and Lemma~\ref{lem.cN} for $A_5$ leads to $\|\P_c c-c_I\|_{1,\cT} \lesssim h^{k+1}$. This and \eqref{eqn.tri} concludes the proof of first estimate $(a)$. \qed

\medskip


\medskip

\noindent {\it Proof of $(b)$.} For the $L^2$ estimate, we follow the Aubin-Nitsche duality arguments. Let $\psi \in H^1(\O)$ solves the following adjoint problem.
\begin{align}\label{eqn.dual}
-\divc(D(\bu)\nabla \psi)-\bu \cdot \nabla\psi+(q^-+1)\psi&=c-\P_c c\mbox{ in } \O, \quad\\
 D(\bu)\nabla \psi\cdot \bn&=0 \mbox{ in } \partial \O.
\end{align}
Since $\O$ is convex, $\psi \in H^2(\O)$ and 
\begin{equation}\label{eqn.aprioridual}
\|\psi\|_2 \lesssim \|c-\P_c c\|.
\end{equation}
 The dual problem \eqref{eqn.dual}, integration by parts, the relation \[(\bu\cdot\nabla \psi,z_h)_{0,K}=(z_h\bu,\nabla \psi)_{0,K}=\half[(\bu\cdot\nabla \psi,z_h)_{0,K}-(\nabla \cdot \bu, \psi z_h)_{0,K}-(\psi,\bu\cdot\nabla z_h)_{0,K}+(\psi \bu \cdot n_e,z_h)_{0,\partial K}]\] for $z_h \in Z_h$, and introduction of the interpolant $\psi_I$ of $\psi$ show
\begin{align}
\|c-\P_c c\|^2
&=(\P_c c-c, \bu \cdot \nabla\psi)-(c-\P_c c,\divc(D(\bu)\nabla \psi))+(c-\P_c c,(q^-+1)\psi)\\
&=\half[(\P_c c-c,\bu \cdot \nabla \psi)-\sum_{K\in \cT_h}(\psi, \bu \cdot \nabla(\P_c c-c))_{0,K}-((q^+-q^-), (\P_c c-c)\psi)]\\
&\qquad +\cD_{\pw}^\bu(\psi,c-\P_c c)-\cN_h(\bu;\psi,c-\P_c c)+(c-\P_c c,(q^-+1)\psi)\\
&=\half[(\P_c c-c,\bu \cdot \nabla \psi)-\sum_{K\in \cT_h}(\psi, \bu \cdot \nabla(\P_c c-c))_{0,K}-((q^++q^-), (\P_c c-c)\psi)]\\
&\qquad +\cD_{\pw}^\bu(\psi,c-\P_c c)-\cN_h(\bu;\psi,c-\P_c c)+(c-\P_c c,\psi)\\
&= \cD_{\pw}^\bu(c-\P_c c,\psi)+\Theta_{\pw}^\bu(c-\P_c c,\psi)+(c-\P_c c,\psi)-\cN_h(\bu;\psi,c-\P_c c)\\
&= (\cD_{\pw}^\bu(c-\P_c c,\psi-\psi_I)+\Theta_{\pw}^\bu(c-\P_c c,\psi-\psi_I)+(c-\P_c c,\psi-\psi_I))\\
&\qquad+(\cD_{\pw}^\bu(c-\P_c c,\psi_I)+\Theta_{\pw}^\bu(c-\P_c c,\psi_I)+(c-\P_c c,\psi_I) -\cN_h(\bu;\psi,c-\P_c c))\\
&=:B_1+B_2.\label{eqn.b1b2}
\end{align}
A generalised \Holder inequality, the estimate $(a)$, \eqref{eqn.cI} for $\psi \in H^2(\O)$ and \eqref{eqn.aprioridual} provide
\begin{equation}\label{eqn.b1}
B_1 \lesssim h^{k+2}\|\psi\|_2 \lesssim h^{k+2}\|c-\P_c c\|.
\end{equation}
The definition of projection in \eqref{eqn.projectionP} implies
\begin{align}
B_2&=\cD_{\pw}^\bu(c,\psi_I)+\Theta_{\pw}^\bu(c,\psi_I)+(c,\psi_I)-\cD_{\pw}^\bu(\P_c c,\psi_I)-\Theta_{\pw}^\bu(\P_c c,\psi_I)-(\P_c c,\psi_I)\\
&\qquad  -\cN_h(\bu;\psi,c-\P_c c)\\
&=\Gamma_{c,\pw}(\bu;c,\psi_I)-\cD_{\pw}^\bu(\P_c c,\psi_I)-\Theta_{\pw}^\bu(\P_c c,\psi_I)-(\P_c c,\psi_I) -\cN_h(\bu;\psi,c-\P_c c)\\
&=\Gamma_{c,h}(\bu;\P_c c,\psi_I)+\cN_h(\bu;c,\psi_I)-\cD_{\pw}^\bu(\P_c c,\psi_I)-\Theta_{\pw}^\bu(\P_c c,\psi_I)-(\P_c c,\psi_I)\\
&\qquad  -\cN_h(\bu;\psi,c-\P_c c)\\
&=(\cD_h^\bu(\P_c c,\psi_I)-\cD_{\pw}^\bu(\P_c c,\psi_I))+(\Theta_h^\bu(\P_c c,\psi_I)-\Theta_{\pw}^\bu(\P_c c,\psi_I))+((\P_c c,\psi_I)_h-(\P_c c,\psi_I))\\
&\qquad +(\cN_h(\bu;c,\psi_I) -\cN_h(\bu;\psi,c-\P_c c))=:B_{2,1}+B_{2,2}+B_{2,3}+B_{2,4}.\label{eqn.b2}
\end{align}
Arguments analogous to the proof of \cite[(5.48),(5.49)]{VEM_general_2016} with $p_h=\P_c c$ and $p=c$ leads to
\begin{align}
B_{2,1}&+B_{2,2}+B_{3,2} \lesssim h^{k+2}\|c-\P_c c\|.
\end{align}
Since $\psi \in H^2(\O)$, $\cN_h(\bu;c,\psi)=0$. This, Lemma~\ref{lem.cN}, \eqref{eqn.cI}, $(a)$, and \eqref{eqn.aprioridual} show
\begin{align}
B_{2,4}&=\cN_h(\bu;c,\psi-\psi_I) -\cN_h(\bu;\psi,c-\P_c c) \\
&\lesssim h^{k+1}(\|D(\bu)\nabla c\|_{k+1}|\psi-\psi_I|_{1,h}+\|c\bu\|_{k+1}\|\psi-\psi_I\|)\\
&\qquad +h(\|D(\bu)\psi\|_1|c-\P_c c|_{1,h}+\|\psi\bu\|_1\|c-\P_c c\|)\\
&\lesssim h^{k+2}\|\psi\|_2 \lesssim h^{k+2}\|c-\P_c c\|.
\end{align}
A combination of the above estimates in \eqref{eqn.b2} leads to $B_2 \lesssim h^{k+2}\|c-\P_c c\|$. This and \eqref{eqn.b1} in \eqref{eqn.b1b2} concludes the proof. \qed

\smallskip

\noindent Differentiation of \eqref{eqn.projectionP} with respect to time and analogous arguments as in Lemma~\ref{lem.cPcerror} yield the following result.
\begin{cor}\label{cor.ctPcterror}
Provided that the continuous data and solution are sufficiently regular in space and time, it holds
	\begin{align*}
\displaystyle	(a)\,	\left\|\frac{\partial}{\partial t}(c-\P_c c)\right\|_{1,\cT_h} \le h^{k+1}\xi_{1,t}, \qquad & (b)\,\left\|\frac{\partial}{\partial t}(c-\P_c c)\right\| \le h^{k+2}\xi_{0,t},
	\end{align*}
	where the constants $\xi_{0,t},\xi_{1,t}>0$ are independent of $h$.
\end{cor}

\noindent The proof of the below error estimate is provided in Section \ref{sec:appendix}, which follow by adapting the corresponding proof in \cite{Veiga_miscibledisplacement_2021} to account for the nonconforming VEM approach, mainly \eqref{eqn.projectionP}.
\begin{thm}\label{thm.c}
Under the mesh assumption (\textbf{D1})-(\textbf{D3}) and the assumption that the continuous data and solution are sufficiently regular in space and time, it holds
	\begin{equation}
		\|c^n-c_h^n\|\lesssim\|c_{0,h}-c^0\|+h^{k+1}+\tau,
	\end{equation}
	where $\lesssim $ depends on $\bu, c,q^+,q^-,\frac{\partial \bu}{\partial t},\frac{\partial^2 \bu}{\partial t^2},\frac{\partial c}{\partial t},$ and $\frac{\partial^2 c}{\partial t^2}.$
\end{thm}

\noindent Consequently, the combination of Theorem~\ref{thm.c} with Theorem~\ref{thm.up} leads to the main result stated in \eqref{eqn.mainresult}.
\section{Numerical results}\label{sec.numericalresults}

This section presents a few examples on general polygonal meshes for the lowest order case $k=0$ to illustrate the theoretical estimates in the previous section. These experiments are conducted on both an ideal test case (Example \ref{sec.example1} and \ref{sec.example2}) and a more realistic test case (Example \ref{sec.example3}), providing comprehensive validation. An interesting aspect of VEM is its ability to be implemented solely based on the degrees of freedom and the polynomial component of the approximation space, see \cite{Veiga_hitchhikersVEM} for details on the implementation procedure.

\smallskip

\noindent  The model problem in Example 1  (resp. 2)  is constructed in such a way that the exact solution is known.  Let the errors in $L^2(\O)$ norm be denoted by
\begin{align*}
\err(p):=\|p^n-\Pi_0p_h^n\|_{},\, \err(\bu):=\|\bu^n-\bPi_0\bu_h^n\|_{},\, \mbox{ and } \err(c):=\|c^n-\Pi_1c_h^n\|_{},
\end{align*} where ($p^n, \bu^n, c^n$) (resp. ($\Pi_0p_h^n, \bPi_0\bu_h^n, \Pi_1 c_h^n$)) is the exact (resp. numerical) solution at the final time $t_n=T$.

		\begin{figure}[h!!]
	\begin{center}
		\begin{minipage}[b]{0.25\linewidth}
			{\includegraphics[width=5cm]{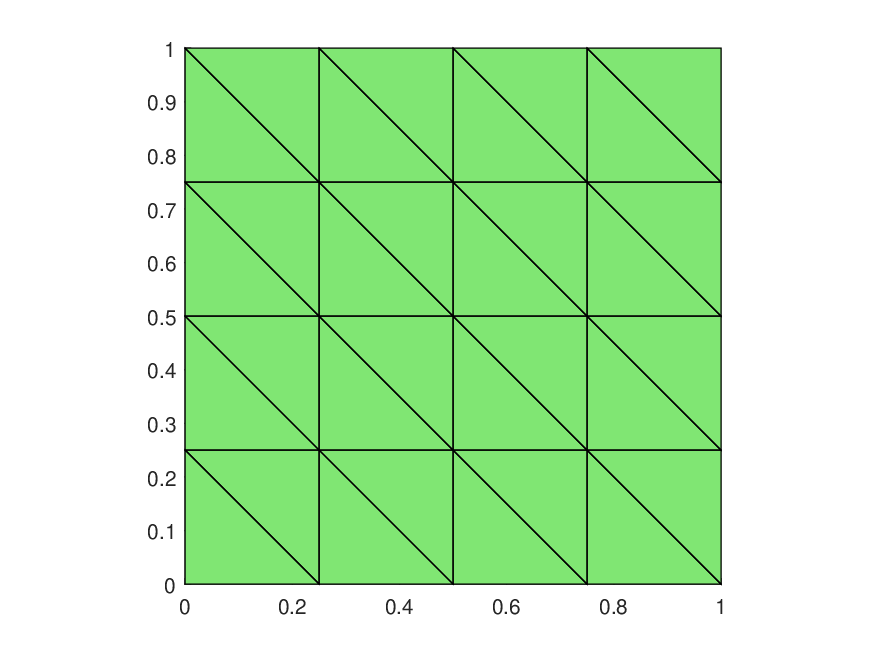}}
			\caption{Triangular Mesh}\label{fig.Triangle}
		\end{minipage}
		\begin{minipage}[b]{0.25\linewidth}
		{\includegraphics[width=5cm]{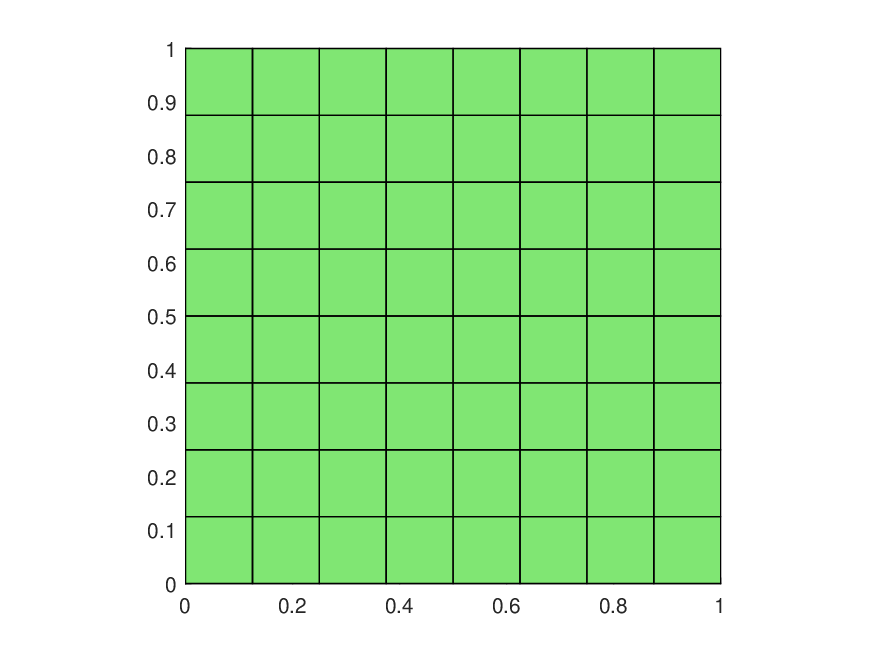}}
		\caption{Square Mesh}\label{fig.Square}
	\end{minipage}
		\begin{minipage}[b]{0.35\linewidth}
			{\includegraphics[width=5cm]{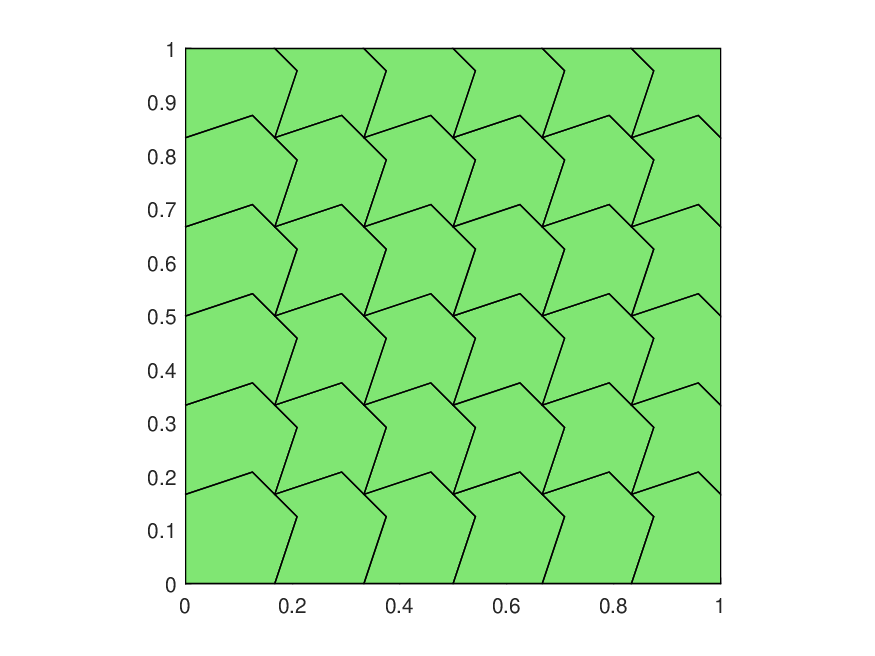}}
				\caption{Concave Mesh}\label{fig.concave}
		\end{minipage}
		\begin{minipage}[b]{0.35\linewidth}
			{\includegraphics[width=5cm]{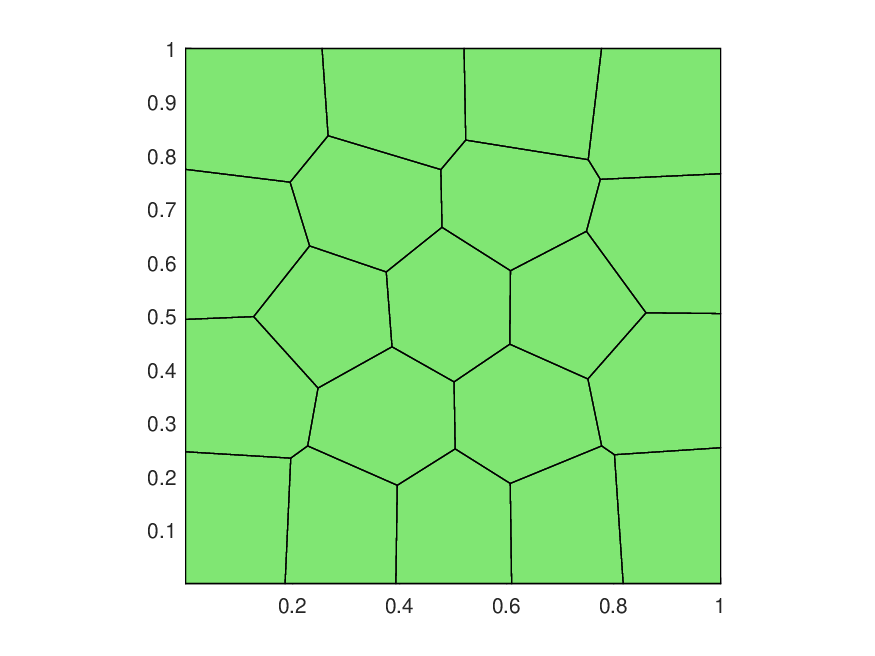}}
				\caption{Structured Voronoi Mesh}\label{fig.SV}
		\end{minipage}
	\begin{minipage}[b]{0.35\linewidth}
	{\includegraphics[width=5cm]{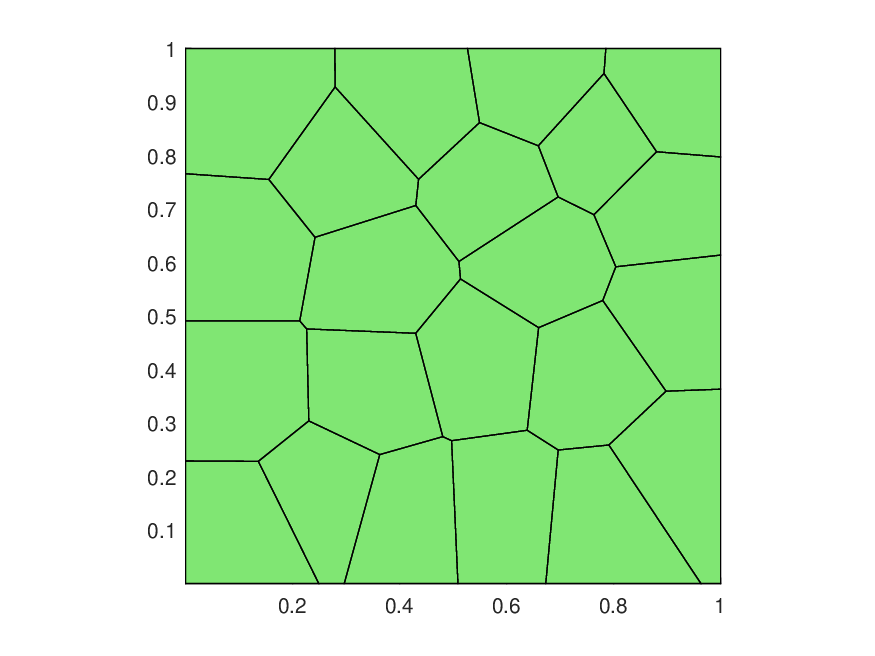}}
		\caption{Random Voronoi Mesh}\label{fig.RV}
\end{minipage}
	\end{center}
\end{figure}

\subsection{Example 1}\label{sec.example1}
\noindent Let the computational domain be $\Omega=(0,1)^2$ and consider the following generalised form of \eqref{eqn.model} taken from \cite[Section 5]{HuFuZhou_2019}:
\begin{subequations}\label{eqn.generalmodel}
\begin{align}
&\displaystyle\phi\frac{\partial c}{\partial t}+\bu\cdot \nabla c -\divc(D(\bu)\nabla c)=f(c) \label{eqn.generalmodela}\\
&\divc\,\bu=q\label{eqn.generalmodelb}\\
&\bu=-a(c)(\nabla p-\bg(c)),\label{eqn.generalmodelc}
\end{align}
\end{subequations}
with the boundary conditions \eqref{eqn.bc} and \eqref{eqn.ic} respectively. The exact solution is given by 
\begin{align*}
c(x,y,t)&=t^2(x^2(x-1)^2+y^2(y-1)^2)\\
\bu(x,y,t)&=2t^2\begin{pmatrix}
x(x-1)(2x-1)\\
y(y-1)(2y-1)
\end{pmatrix}\\
p(x,y,t)&=\frac{-1}{2}c^2-2c+\frac{17}{6300}t^4+\frac{2}{15}t^2.
\end{align*}
As in \cite{HuFuZhou_2019, Veiga_miscibledisplacement_2021}, choose $D(\bu)=|\bu|+0.02$, where $d_m=0.02, d_\ell=d_t=1$, $a(c)=c+2$, $\phi=1$, $ \bg(c)=0$ and $c_0=0$.   Then the right hand side load functions $f$ and $q$ can be computed from \eqref{eqn.generalmodela} and \eqref{eqn.generalmodelb}.

 \noindent A series of triangular, square, concave, structured Voronoi, and random Voronoi meshes (see Figure~\ref{fig.Triangle}-\ref{fig.RV}) are employed to test the convergence results for the VEM. The alternative (or diagonal) stabilisations in \cite{Veiga_miscibledisplacement_2021} were also examined for these meshes, yielding similar results; thus, it has been excluded from further consideration.

\smallskip

\noindent In the case of triangular, square, and concave meshes, the mesh size $h$ undergoes a reduction by a factor of 2 during each refinement. For the remaining two meshes, the mesh size is approximately halved with each subsequent refinement. Hence, time-step size $\tau$ is initially set to be $\tau=T/5$ (at the first level) and is subsequently halved at each subsequent level.
\begin{table}[h!!]
	\caption{\small{Convergence results, Example 1, Triangular mesh}}
	{\small{\footnotesize					\begin{center}
				\begin{tabular}{|c|c ||c
						|c||c | c ||c|c|}
					\hline
					$h$&$\tau$&$\err(\bu)$ & Order  & $\err(p)$ & Order  &$\err(c)$ & Order  \\ 
					\hline
0.707107 & 0.001000 & 0.797130 & -& 0.833294 & -& 0.159478 & -\\
0.353553 & 0.000500 & 0.503730 & 0.6622 & 0.369737 & 1.1723 & 0.077156 & 1.0475\\
0.176777 & 0.000250 & 0.265467 & 0.9241 & 0.185240 & 0.9971 & 0.036182 & 1.0925\\
0.088388 & 0.000125 & 0.134444 & 0.9815 & 0.093008 & 0.9940 & 0.017820 & 1.0217\\
0.044194 & 0.000063 & 0.067436 & 0.9954 & 0.046564 & 0.9981 & 0.008898 & 1.0020
					\\	\hline				
				\end{tabular}
		\end{center}}					
	}\label{table.eg1.Triangle}	
\end{table}

\begin{table}[h!!]
	\caption{\small{Convergence results, Example 1, Square mesh}}
	{\small{\footnotesize					\begin{center}
				\begin{tabular}{|c|c ||c
						|c||c | c ||c|c|}
					\hline
					$h$&$\tau$&$\err(\bu)$ & Order  & $\err(p)$ & Order  &$\err(c)$ & Order  \\ 
					\hline
0.353553 & 0.002000& 0.480580 & - & 0.447421 & - & 0.287735 &-\\
0.176777 & 0.001000& 0.235625 & 1.0283 & 0.226155 & 0.9843 & 0.141870 &1.0202 \\
0.088388 & 0.000500& 0.117153 & 1.0081 & 0.113820 & 0.9906 & 0.070718 &1.0044 \\
0.044194 & 0.000250& 0.058492 & 1.0021 & 0.057017 & 0.9973 & 0.035336 &1.0009 \\
0.022097 & 0.000125& 0.029235 & 1.0005 & 0.028523 & 0.9993 & 0.017666 &1.0002
					\\	\hline				
				\end{tabular}
		\end{center}}					
	}\label{table.eg1.Square}	
\end{table}

\begin{table}[h!!]
	\caption{\small{Convergence results, Example 1, Concave mesh}}
	{\small{\footnotesize					\begin{center}
				\begin{tabular}{|c|c ||c
						|c||c | c ||c|c|}
					\hline
					$h$&$\tau$&$\err(\bu)$ & Order  & $\err(p)$ & Order  &$\err(c)$ & Order  \\ 
					\hline
0.485913 & 0.002000 & 0.667205 & - & 0.629865 & - & 0.300115 & - \\
0.242956 & 0.001000 & 0.323976 & 1.0422 & 0.311893 & 1.0140 & 0.143970 & 1.0597 \\
0.121478 & 0.000500 & 0.161591 & 1.0035 & 0.156378 & 0.9960 & 0.070986 & 1.0202 \\
0.060739 & 0.000250 & 0.080718 & 1.0014 & 0.078337 & 0.9973 & 0.035369 & 1.0051 \\
0.030370 & 0.000125 & 0.040297 & 1.0022 & 0.039196 & 0.9990 & 0.017670 & 1.0012
					\\	\hline				
				\end{tabular}
		\end{center}}					
	}\label{table.eg1.Concave}	
\end{table}

\begin{table}[h!!]
	\caption{\small{Convergence results, Example 1, Structured Voronoi mesh}}
	{\small{\footnotesize					\begin{center}
				\begin{tabular}{|c|c ||c
						|c||c | c ||c|c|}
					\hline
					$h$&$\tau$&$\err(\bu)$ & Order  & $\err(p)$ & Order  &$\err(c)$ & Order  \\ 
					\hline
0.707107 & 0.002000 & 1.000000 & - & 0.999985 & -&0.290645 &-\\
0.340697 & 0.001000 & 0.433758 & 1.1439 & 0.395954 & 1.2688&0.149120 & 0.9139
\\
0.171923 & 0.000500 & 0.215287 & 1.0242
& 0.206368 & 0.9528&0.071650 & 1.0717\\
0.083555 & 0.000250 & 0.107060 & 0.9682& 0.104157 & 0.9476&0.035457 & 0.9749\\
0.047445 & 0.000125 & 0.058463 & 1.0690 & 0.056755 & 1.0729&0.017688 & 1.2288\\
0.027786 & 0.000063 & 0.033706 & 1.0293 & 0.032840 & 1.0225 & 0.008838 & 1.2968 
					\\	\hline				
				\end{tabular}
		\end{center}}					
	}\label{table.eg1.SV}	
\end{table}

\begin{table}[h!!]
	\caption{\small{Convergence results, Example 1, Random Voronoi mesh}}
	{\small{\footnotesize					\begin{center}
				\begin{tabular}{|c|c ||c
						|c||c | c ||c|c|}
					\hline
					$h$&$\tau$&$\err(\bu)$ & Order  & $\err(p)$ & Order  &$\err(c)$ & Order  \\ 
					\hline
0.736793 & 0.002000 & 1.002026 & - & 0.998901 & - & 0.292073 & - \\
0.373676 & 0.001000 & 0.434253 & 1.2316 & 0.412104 & 1.3041 & 0.148433 & 0.9970\\
0.174941 & 0.000500 & 0.198870 & 1.0290 & 0.187068 & 1.0407 & 0.071415 & 0.9640 \\
0.089478 & 0.000250 & 0.094805 & 1.1050& 0.086812 & 1.1451 & 0.035408 & 1.0464 \\
0.041643 & 0.000125 & 0.040431 & 1.1142 & 0.037884 & 1.0841 & 0.017670 & 0.9087\\
	\hline				
				\end{tabular}
		\end{center}}					
	}\label{table.eg1.RV}	
\end{table}
\noindent Table~\ref{table.eg1.Triangle}-\ref{table.eg1.RV} show  errors and orders of convergence for the velocity $\bu$, pressure $p$, and concentration $c$ for the aforementioned five types of meshes. Observe that linear order of convergences are obtained for these variables in $L^2$ norm. These numerical order of convergence clearly matches the expected order of convergence given in \eqref{eqn.mainresult}, Theorems~\ref{thm.up}, and \ref{thm.c} with ($k=0$) respectively.
\subsection{Example 2}\label{sec.example2}
In this example, we consider the generalised form \eqref{eqn.generalmodel} with \cite{Veiga_miscibledisplacement_2021}
\begin{align*}
	c(x,y,t)&=t^2(1-\exp(-100(x^2+y^2)))\\
	\bu(x,y,t)&=200t^2\begin{pmatrix}
		x\exp(-100(x^2+y^2))\\
		y\exp(-100(x^2+y^2))
	\end{pmatrix}\\
	p(x,y,t)&=\frac{-1}{2}c^2-2c+\eta_1t^4+\eta_2t^2,
\end{align*}
where $\eta_1,\eta_2$ are the constants such that $p(\cdot,t) \in L^2_0(\Omega)$ for all $t \in [0,T]$. The choice of parameters are same as that in Example \ref{sec.example1}. Observe that the analytical solutions exhibit a corner layer positioned at $(0,0)$ for all time values within the interval $(0,T)$.

\smallskip

\noindent The errors and orders of convergence for the velocity $\bu$, pressure $p$, and concentration $c$ are presented in Table \ref{table.eg2.Triangle}-\ref{table.eg2.RV}. The orders of convergence results are similar to those obtained in Example \ref{sec.example1}. 

\begin{table}[h!!]
	\caption{\small{Convergence results, Example 2, Triangular mesh}}
	{\small{\footnotesize					\begin{center}
				\begin{tabular}{|c|c ||c
						|c||c | c ||c|c|}
				\hline
				$h$&$\tau$&$\err(\bu)$ & Order  & $\err(p)$ & Order  &$\err(c)$ & Order  \\ 
				\hline
	0.353553 & 0.000500 & 0.925949 & - & 0.731189 & - & 0.072995 & - \\
	0.176777 & 0.000250 & 0.583976 & 0.6650 & 0.324370 & 1.1726 & 0.036434 & 1.0025 \\
	0.088388 & 0.000125 & 0.358792 & 0.7028 & 0.164764 & 0.9772 & 0.017865 & 1.0281 \\
	0.044194 & 0.000063 & 0.185613 & 0.9508 & 0.085695 & 0.9431 & 0.008913 & 1.0031 \\
	0.022097 & 0.000031 & 0.093620 & 0.9874 & 0.043222 & 0.9875 & 0.004508 & 0.9834
				\\	\hline				
			\end{tabular}
	\end{center}}					
}\label{table.eg2.Triangle}	
\end{table}

\begin{table}[h!!]
\caption{\small{Convergence results, Example 2, Square mesh}}
{\small{\footnotesize					\begin{center}
			\begin{tabular}{|c|c ||c
					|c||c | c ||c|c|}
			\hline
			$h$&$\tau$&$\err(\bu)$ & Order  & $\err(p)$ & Order  &$\err(c)$ & Order  \\ 
			\hline
0.353553 & 0.002000 & 1.002352 & - & 0.866949 & - & 0.286989 & - \\
0.176777 & 0.001000 & 0.730307 & 0.4568 & 0.521653 & 0.7329 & 0.142132 & 1.0138 \\
0.088388 & 0.000500 & 0.378479 & 0.9483 & 0.250530 & 1.0581 & 0.070805 & 1.0053 \\
0.044194 & 0.000250 & 0.183926 & 1.0411 & 0.127739 & 0.9718 & 0.035345 & 1.0023\\
0.022097 & 0.000125 & 0.090710 & 1.0198 & 0.064188 & 0.9928 & 0.017671 & 1.0001
			\\	\hline				
		\end{tabular}
\end{center}}					
}\label{table.eg2.Square}	
\end{table}

\begin{table}[h!!]
\caption{\small{Convergence results, Example 2, Concave mesh}}
{\small{\footnotesize					\begin{center}
		\begin{tabular}{|c|c ||c
				|c||c | c ||c|c|}
		\hline
		$h$&$\tau$&$\err(\bu)$ & Order  & $\err(p)$ & Order  &$\err(c)$ & Order  \\ 
		\hline
0.485913 & 0.002000 & 0.987667 & - & 0.904168 & - & 0.286189 & - \\
0.242956 & 0.001000 & 0.946291 & 0.0617 & 0.771791 & 0.2284 & 0.146601 & 0.9651 \\
0.121478 & 0.000500 & 0.561131 & 0.7539 & 0.372617 & 1.0505 & 0.071455 & 1.0368\\
0.060739 & 0.000250 & 0.268312 & 1.0644 & 0.179758 & 1.0516 & 0.035406 & 1.0131 \\
0.030370 & 0.000125 & 0.129542 & 1.0505 & 0.089631 & 1.0040 & 0.017679 & 1.0020 
		\\	\hline				
	\end{tabular}
	\end{center}}					
}\label{table.eg2.Concave}	
\end{table}

\begin{table}[h!!]
\caption{\small{Convergence results, Example 2, Structured Voronoi mesh}}
{\small{\footnotesize					\begin{center}
	\begin{tabular}{|c|c ||c
			|c||c | c ||c|c|}
	\hline
	$h$&$\tau$&$\err(\bu)$ & Order  & $\err(p)$ & Order  &$\err(c)$ & Order  \\ 
	\hline
0.707107 & 0.002000 & 0.999814 &- & 0.884889 & - & 0.283369 & - \\
0.340697 & 0.001000 & 0.978872 & 0.0290 & 0.826595 & 0.0933 & 0.147909 & 0.8904 \\
0.171923 & 0.000500 & 0.691810 & 0.5075 & 0.491317 & 0.7606 & 0.071902 & 1.0546 \\
0.083555 & 0.000250 & 0.338575 & 0.9903 & 0.223018 & 1.0947 & 0.035535 & 0.9768 \\
0.047445 & 0.000125 & 0.177340 & 1.1427& 0.119303 & 1.1054 & 0.017698 & 1.2317\\
	\hline				
\end{tabular}
\end{center}}					
}\label{table.eg2.SV}	
\end{table}

\begin{table}[h!!]
\caption{\small{Convergence results, Example 2, Random Voronoi mesh}}
{\small{\footnotesize					\begin{center}
\begin{tabular}{|c|c ||c
		|c||c | c ||c|c|}
\hline
$h$&$\tau$&$\err(\bu)$ & Order  & $\err(p)$ & Order  &$\err(c)$ & Order  \\ 
\hline
0.736793 & 0.002000 & 0.995989 & - & 0.889492 & -& 0.283495 & -\\
0.373676 & 0.001000 & 0.984292 & 0.0174 & 0.812099 & 0.1341 & 0.146108 & 0.9763\\
0.174941 & 0.000500 & 0.647665 & 0.5515 & 0.432563 & 0.8300 & 0.071755 & 0.9370\\
0.089478 & 0.000250 & 0.335313 & 0.9819 & 0.215149 & 1.0417 & 0.035519 & 1.0488\\
0.041643 & 0.000125 & 0.134751 & 1.1919 & 0.091517 & 1.1176 & 0.017678 & 0.9122 \\
\hline				
\end{tabular}
\end{center}}					
}\label{table.eg2.RV}	
\end{table}
\subsection{Example 3}\label{sec.example3}
In this example, we examine a more realistic test \cite{cha-07-por,Wang2000miscible, Veiga_miscibledisplacement_2021} aiming to provide a qualitative comparison with the anticipated benchmark solution found in existing literature. Here, consider \eqref{eqn.modela}-\eqref{eqn.modelc} with \eqref{eqn.bc} and \eqref{eqn.ic}, where the
spatial domain is $\Omega=(0,1000)\times (0,1000)$ ft$^2$, the time period is $[0,T]=[0,3600]$ days,
and the viscosity of the oil is $\mu(0)=1.$ cp. The injection well is situated in the upper-right corner of the domain at $(1000, 1000)$, with an injection rate of $q^+ = 30 $ ft$^2/$day and an injection concentration of $\hc = 1.0.$ Meanwhile, the production well is positioned at the lower-left corner at $(0, 0)$ and has a production rate of $q^- = 30$ ft$^2/$day. The initial concentration across the domain is given by $c_0(x,y) = 0$. The parameters are chosen as $\phi=0.1, d_\ell=50, d_t=5, \bg(c)=0$, and $a(c)=k(\bx)(1+(M^{1/4}-1)c)^4$, where $k(\bx)$, $M$ and $d_m$ are mentioned below for each of the four tests. The schemes were tested on a triangular mesh with 2048 elements and on a $64 \times 64$ square mesh, both using a time step size of $\tau=36$ days. 
\subsubsection*{Test 1} Assume that the permeability is $k(\bx)=80$, the mobility ratio between the resident and injected fluid is $M = 1$, and the molecular diffusion is $d_m=10$. Figure~\ref{fig.test1} displays the surface and contour plots of the concentration at $t=3$ years (1080 days) and $t=10$ years (3600 days) for triangular and square meshes. Since $M=1$, $a(c)$ is constant and this implies that the fluid has a constant viscosity $\mu(c)=\mu(0)$, Figure~\ref{fig.test1} shows that the velocity $\bu$ is radial and the contour plots for the concentration is circular until the injecting fluid reaches the production well at $t=3$ years. At $t=10$ years, when these plots are reached at production well, the injecting fluid continues to fill the whole domain until $c=1$.
	\begin{figure}[h!!]
	\begin{center}
			\begin{minipage}[b]{0.23\linewidth}
			{\includegraphics[width=4.5cm]{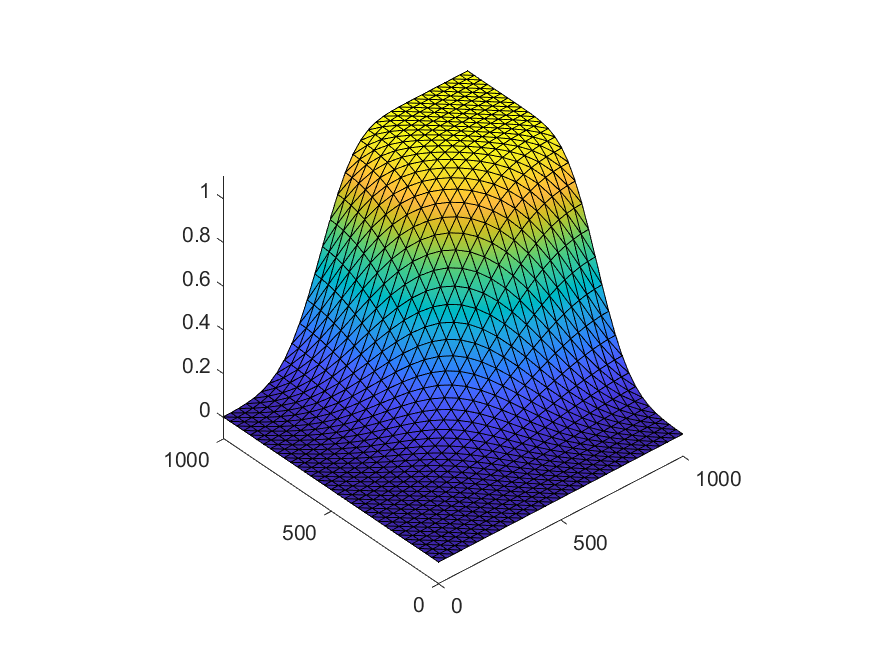}}
		\end{minipage}
		\begin{minipage}[b]{0.2\linewidth}
			{\includegraphics[width=4.2cm]{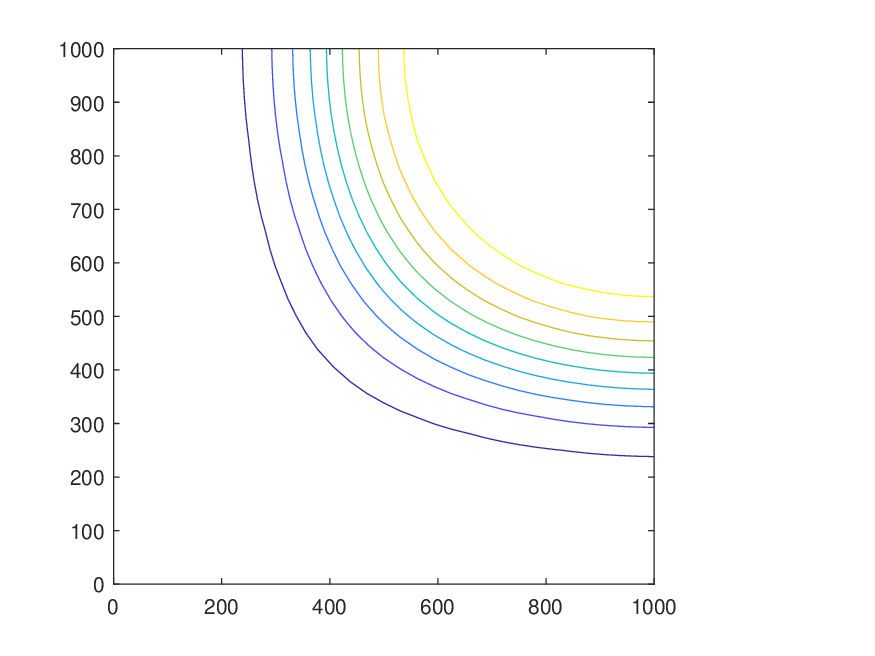}}
		\end{minipage}
		\begin{minipage}[b]{0.23\linewidth}
			{\includegraphics[width=4.5cm]{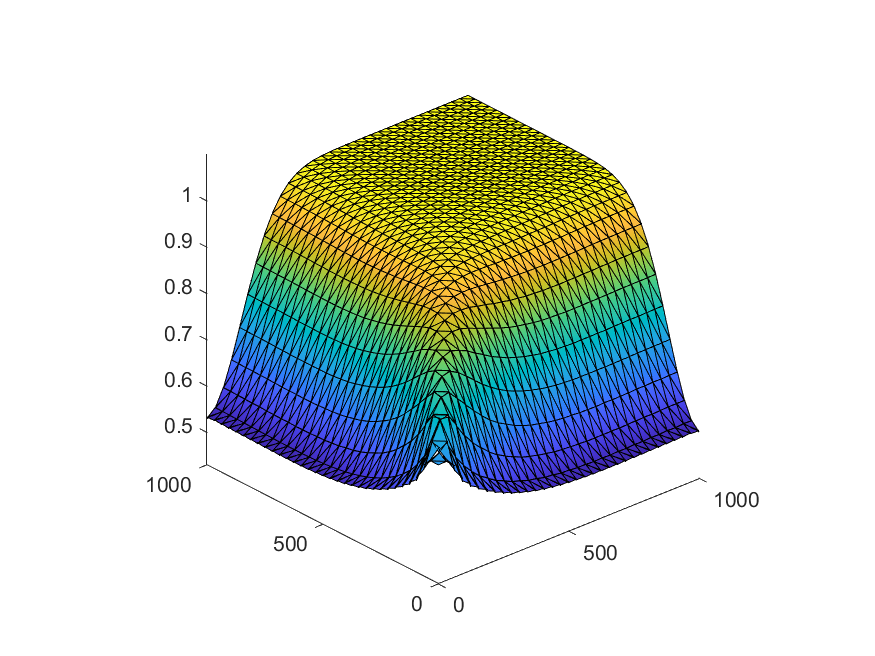}}
		\end{minipage}
		\begin{minipage}[b]{0.23\linewidth}
			{\includegraphics[width=4.2cm]{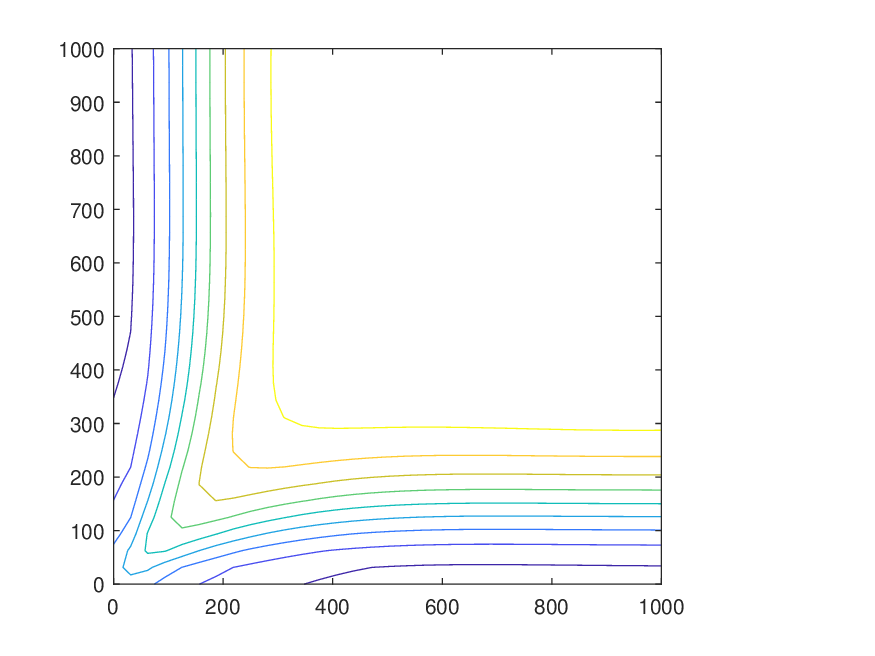}}
		\end{minipage}
		\begin{minipage}[b]{0.23\linewidth}
			{\includegraphics[width=4.5cm]{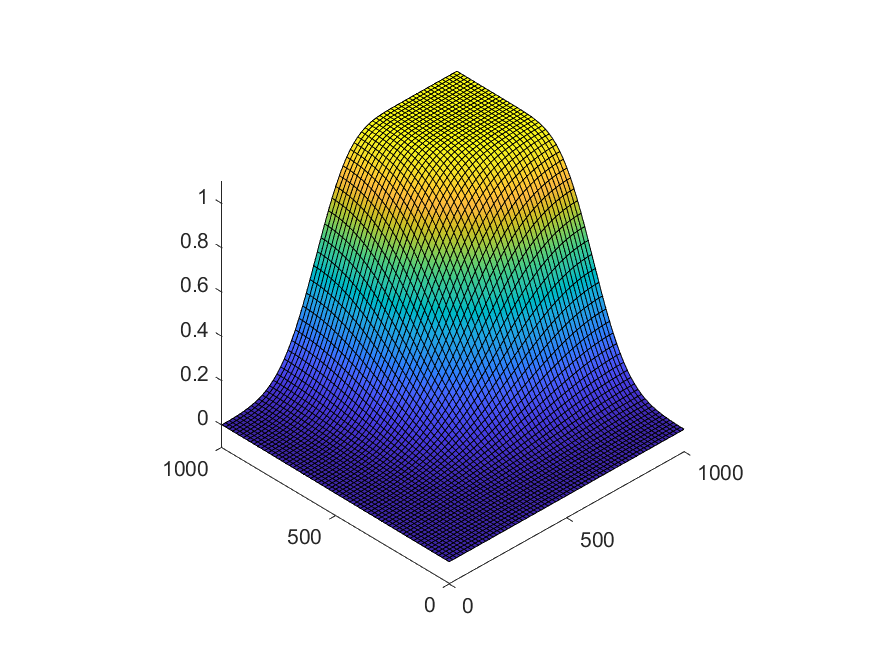}}
		\end{minipage}
		\begin{minipage}[b]{0.2\linewidth}
			{\includegraphics[width=4.2cm]{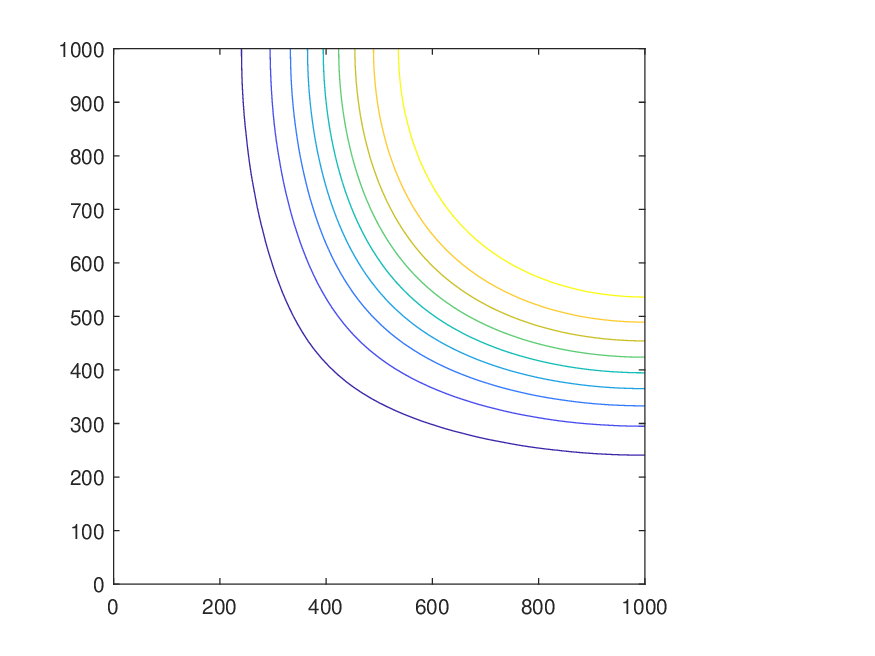}}
		\end{minipage}
			\begin{minipage}[b]{0.23\linewidth}
			{\includegraphics[width=4.5cm]{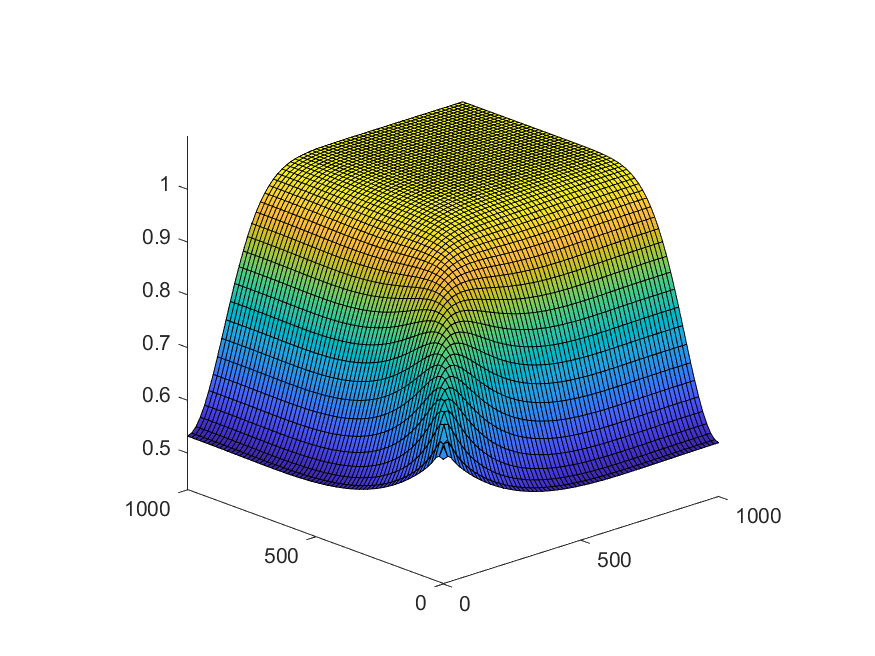}}
		\end{minipage}
		\begin{minipage}[b]{0.23\linewidth}
			{\includegraphics[width=4.2cm]{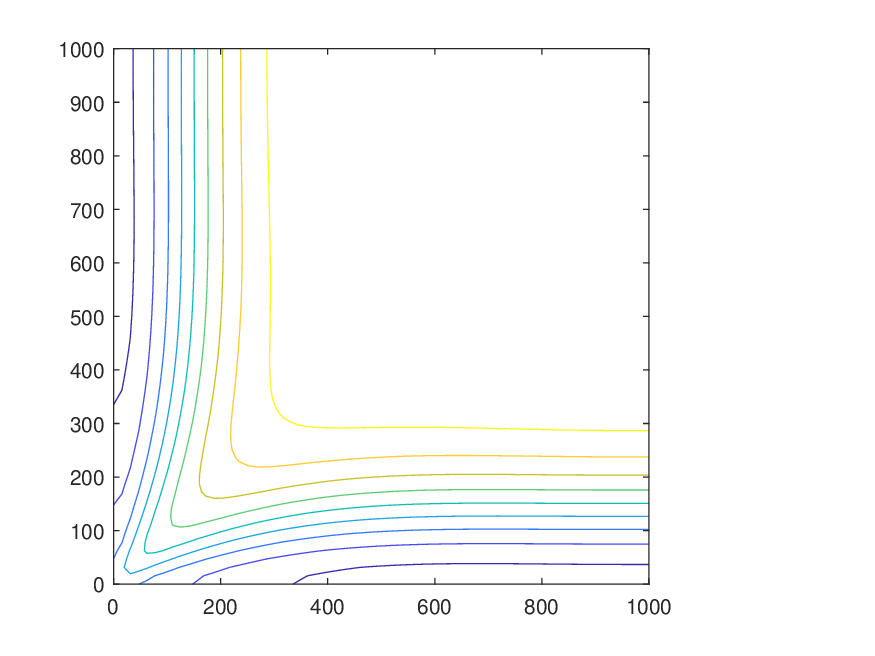}}
		\end{minipage}
		\caption{Surface and contour plots of the concentration in Test 1 at $t=3$ and $t=10$ years for triangular (top) and square (bottom) meshes}\label{fig.test1}
	\end{center}
\end{figure}
\subsubsection*{Test 2} In this test, we take $k(\bx)=80$, $M = 41$ and the molecular diffusion $d_m=0$. The surface and contour plots of the concentration at $t=3$ and $t=10$ years are presented in Figure~\ref{fig.test2}. The viscosity $\mu(c)$ here depends on the concentration $c$ unlike test 1 due to the choice of $M$. This and the difference between the longitudinal and the dispersion coefficients together with the absence of the molecular diffusion imply that the fluid flow is much faster along the diagonal direction, which can be observed from the figure. The results align with those reported in the literature, as demonstrated, for instance, in \cite[Section 6.2]{GDMMiscibleDisplacement}.
	\begin{figure}[h!!]
	\begin{center}
		\begin{minipage}[b]{0.23\linewidth}
			{\includegraphics[width=4.5cm]{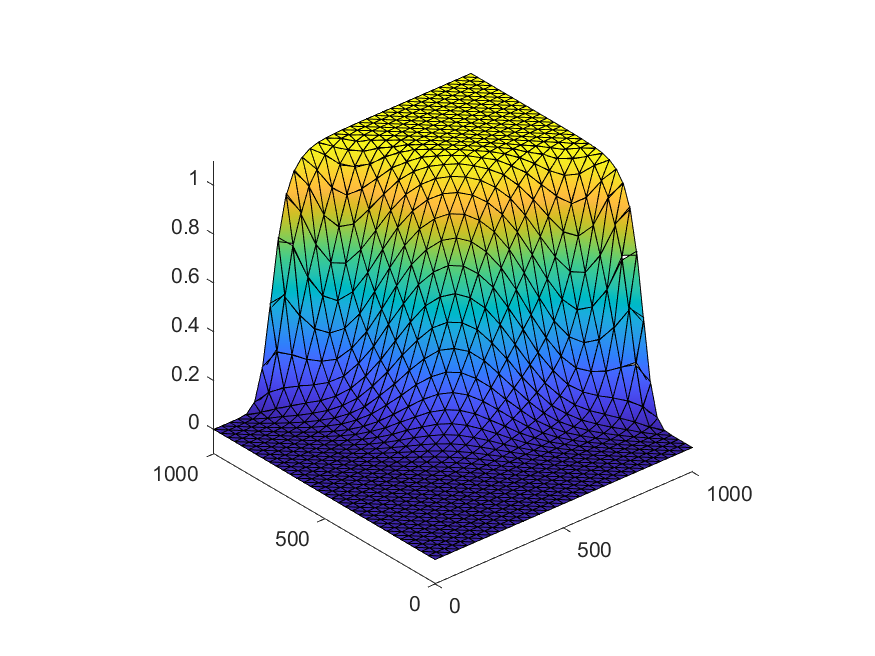}}
		\end{minipage}
		\begin{minipage}[b]{0.2\linewidth}
			{\includegraphics[width=4.2cm]{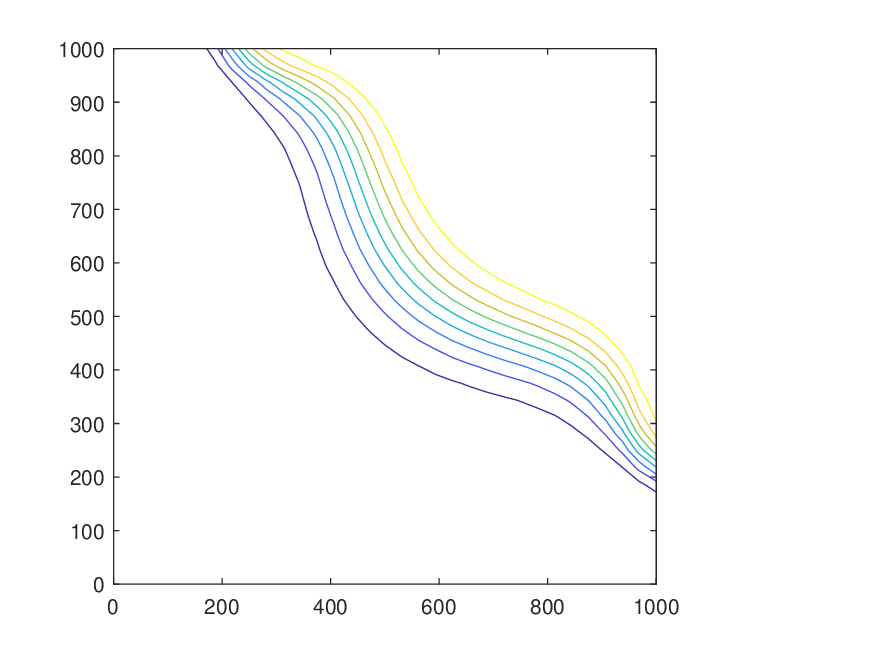}}
		\end{minipage}
		\begin{minipage}[b]{0.23\linewidth}
			{\includegraphics[width=4.5cm]{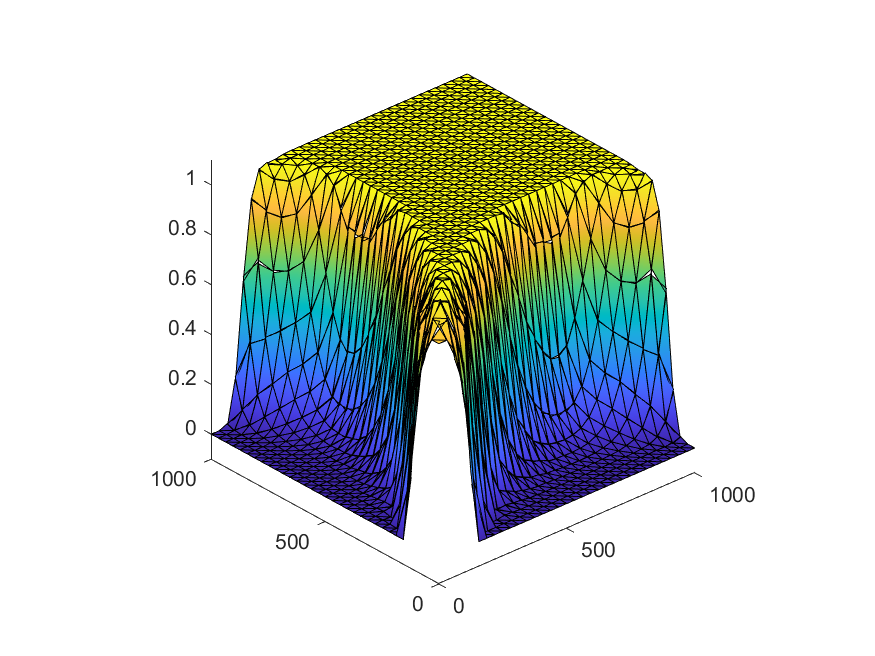}}
		\end{minipage}
		\begin{minipage}[b]{0.23\linewidth}
			{\includegraphics[width=4.2cm]{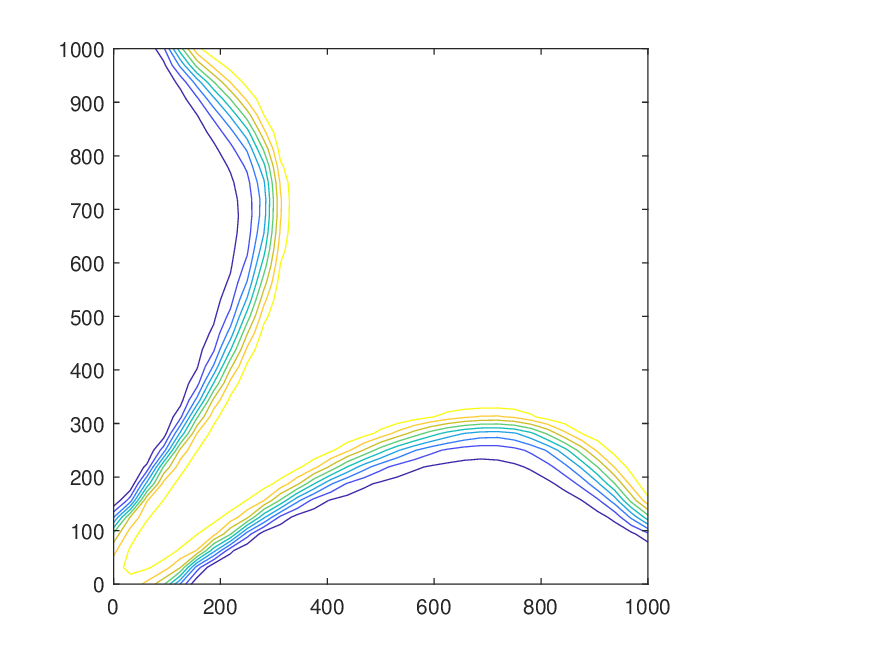}}
		\end{minipage}
		\begin{minipage}[b]{0.23\linewidth}
			{\includegraphics[width=4.5cm]{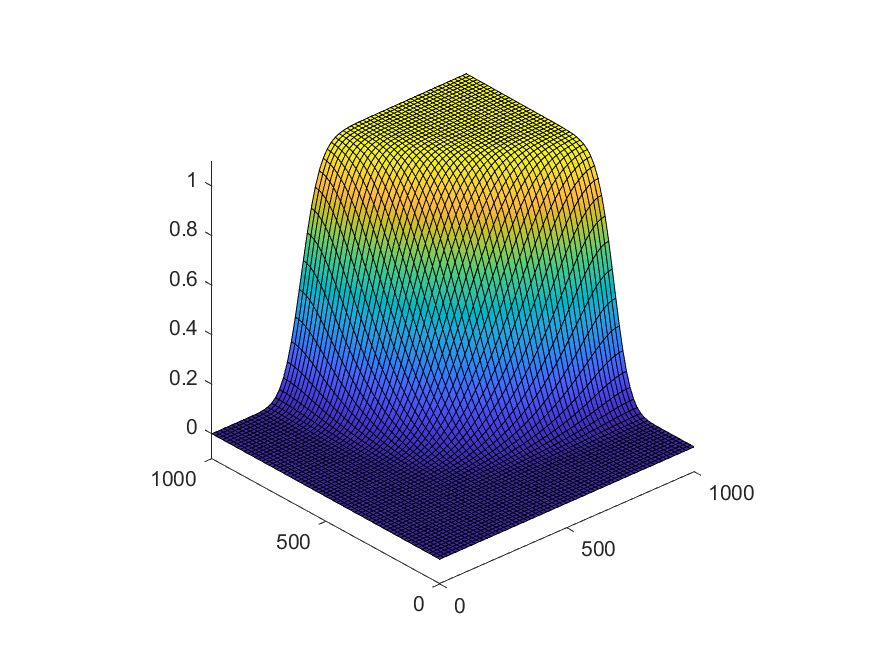}}
		\end{minipage}
		\begin{minipage}[b]{0.2\linewidth}
			{\includegraphics[width=4.2cm]{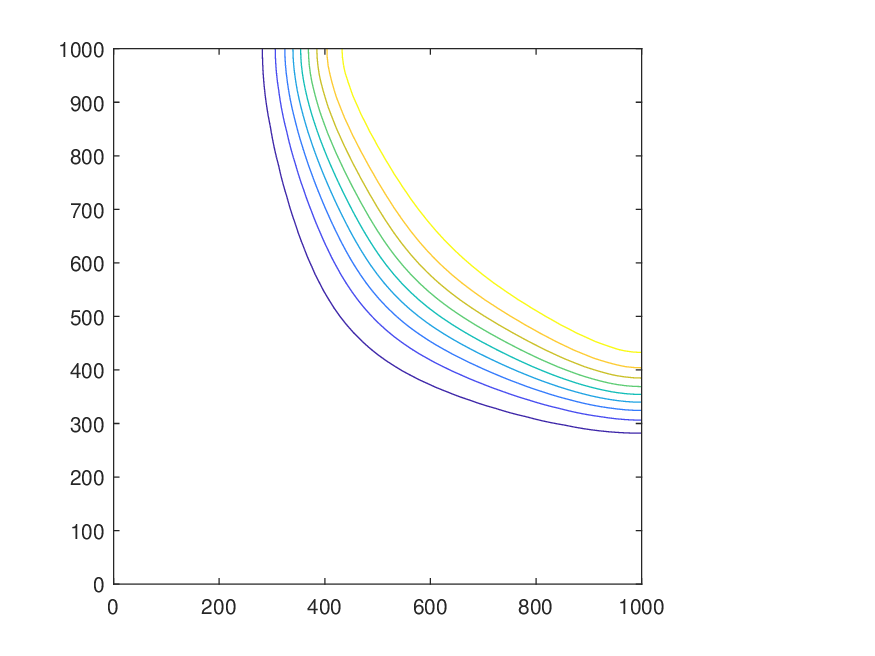}}
		\end{minipage}
		\begin{minipage}[b]{0.23\linewidth}
			{\includegraphics[width=4.5cm]{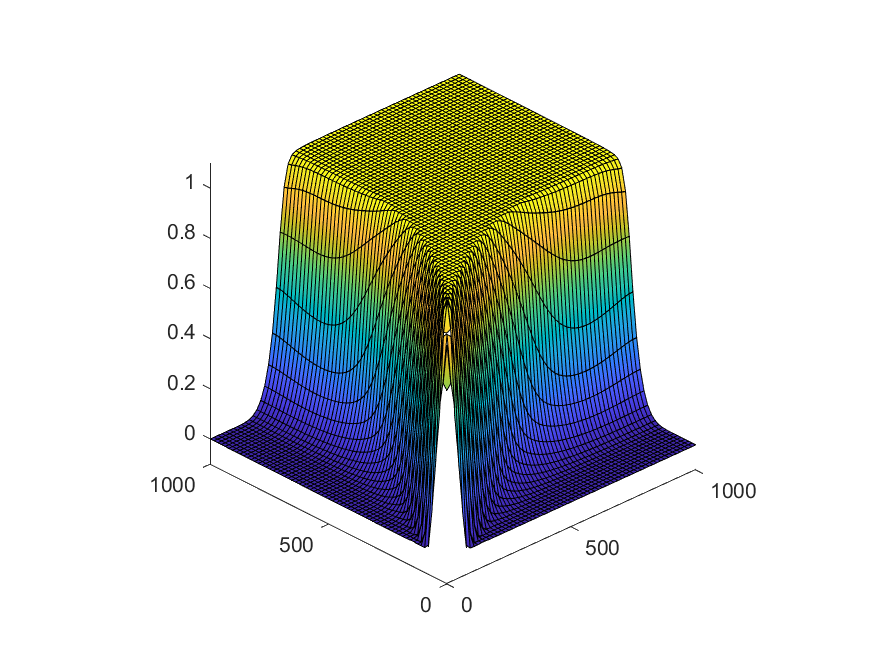}}
		\end{minipage}
		\begin{minipage}[b]{0.23\linewidth}
			{\includegraphics[width=4.2cm]{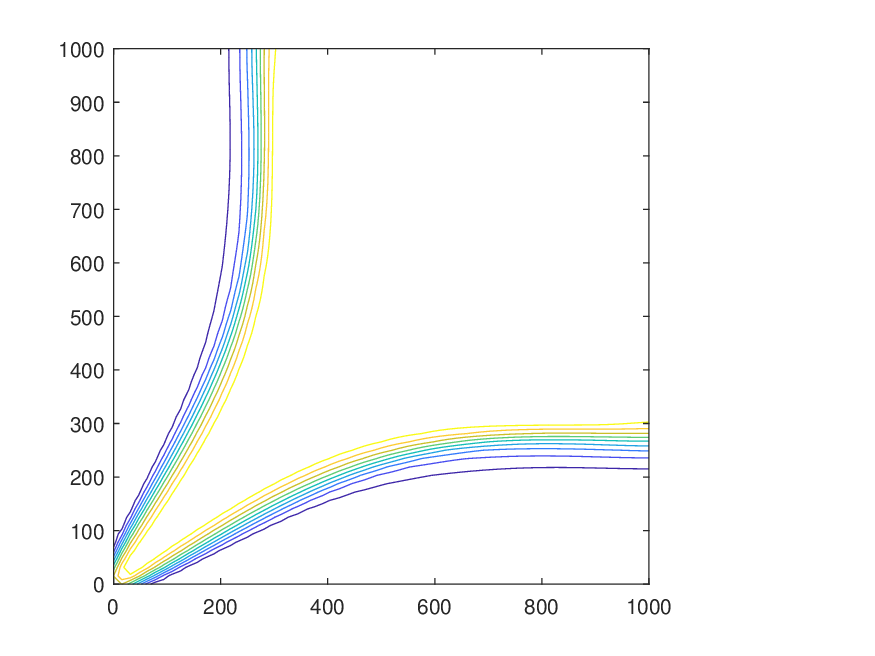}}
		\end{minipage}
		\caption{Surface and contour plots of the concentration in Test 2 at $t=3$ and $t=10$ years for triangular (top) and square (bottom) meshes}\label{fig.test2}
	\end{center}
\end{figure}
\subsubsection*{Test 3} Here, we consider the numerical simulation of a miscible displacement problem with discontinuous permeability, which is commonly encountered in many field applications.. The data is same as given in Test 1 except the permeability
of the medium $k(\bx)$. Let $k(\bx)=80$ on the sub-domain $\Omega_L:=(0,1000) \times (0,500)$ and$k(\bx)=20$ on the sub-domain $\Omega_U:=(0,1000) \times (500,1000)$. The contour and surface plot at $t = 3$ and $t = 10$ years are given in Figure \ref{fig.test3}. 
\begin{figure}[h!!]
	\begin{center}
		\begin{minipage}[b]{0.23\linewidth}
			{\includegraphics[width=4.5cm]{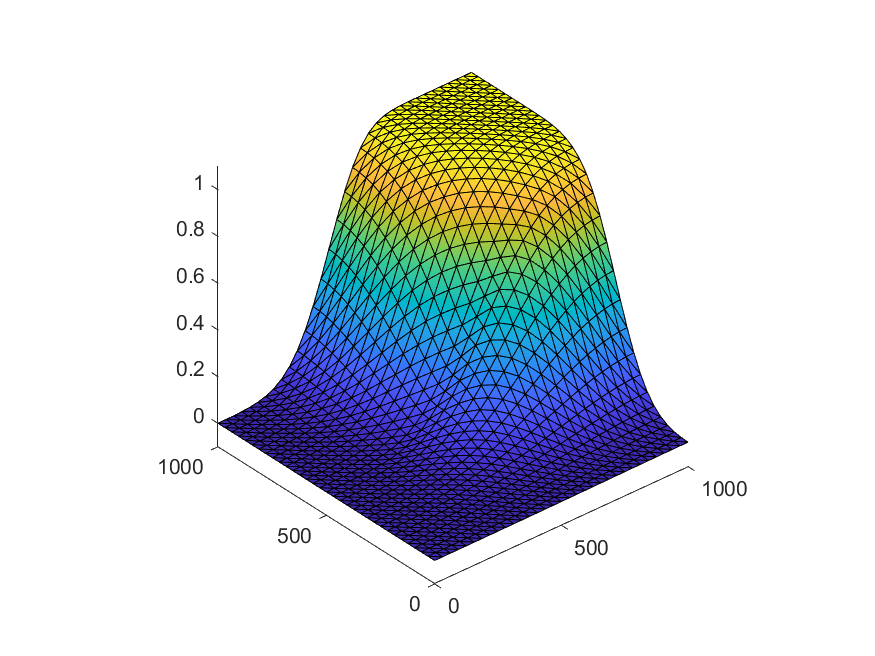}}
		\end{minipage}
		\begin{minipage}[b]{0.2\linewidth}
			{\includegraphics[width=4.2cm]{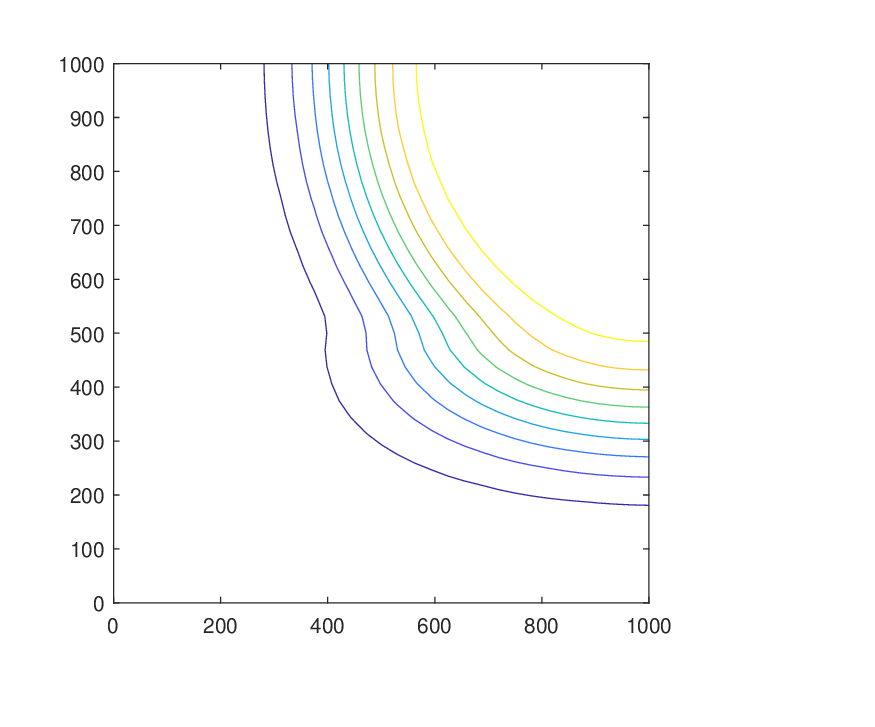}}
		\end{minipage}
		\begin{minipage}[b]{0.23\linewidth}
			{\includegraphics[width=4.5cm]{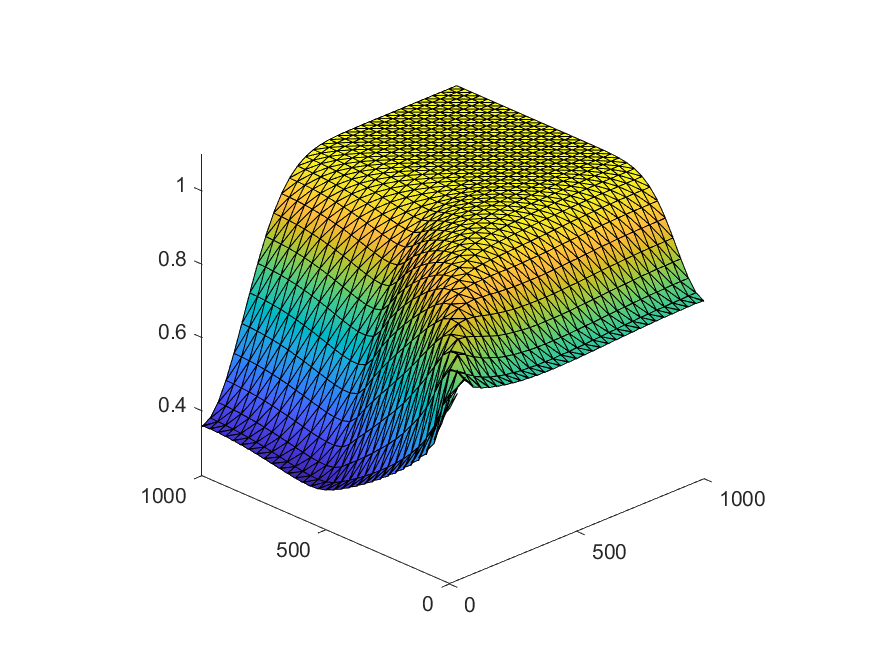}}
		\end{minipage}
		\begin{minipage}[b]{0.23\linewidth}
			{\includegraphics[width=4.2cm]{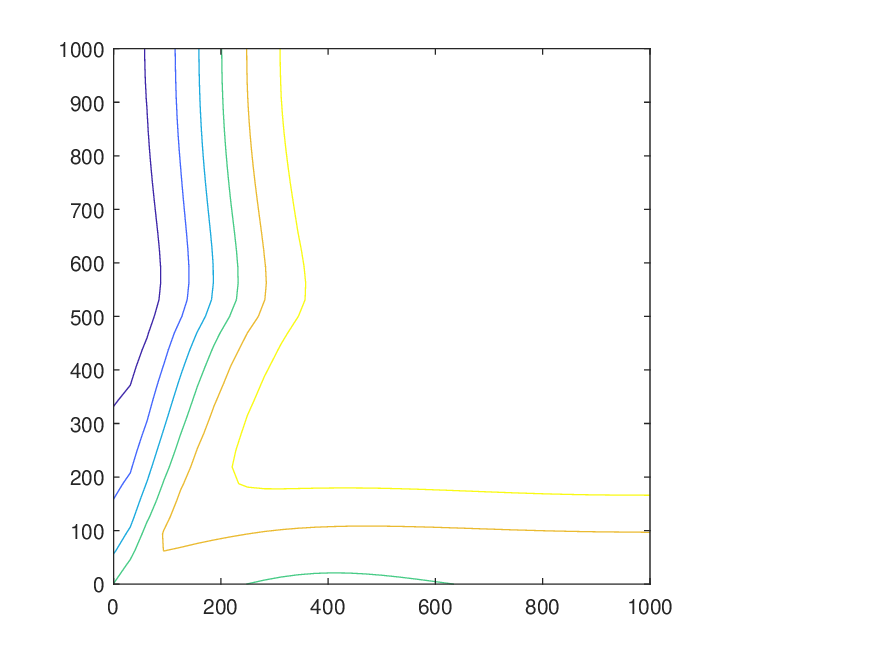}}
		\end{minipage}
		\begin{minipage}[b]{0.23\linewidth}
			{\includegraphics[width=4.5cm]{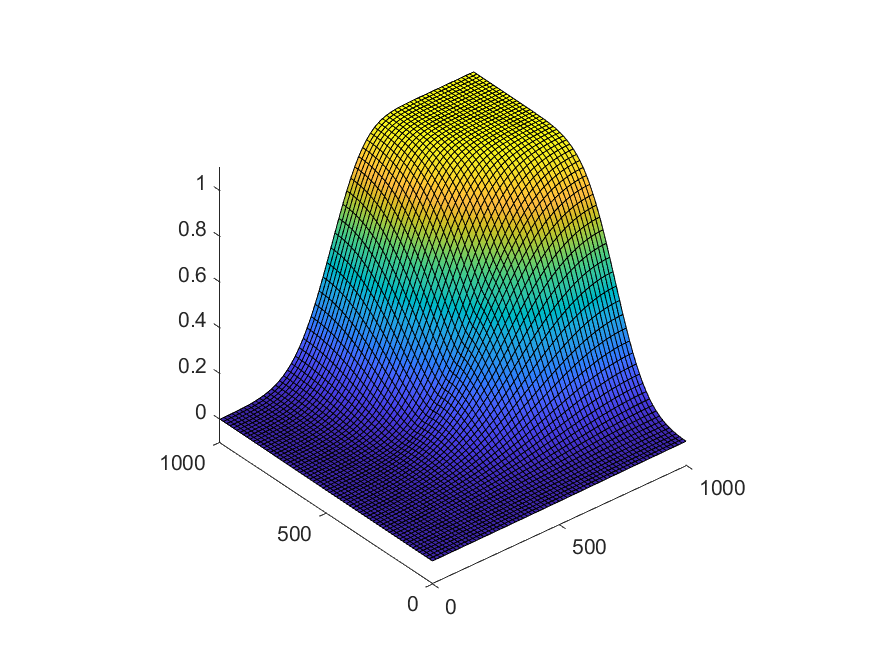}}
		\end{minipage}
		\begin{minipage}[b]{0.2\linewidth}
			{\includegraphics[width=4.2cm]{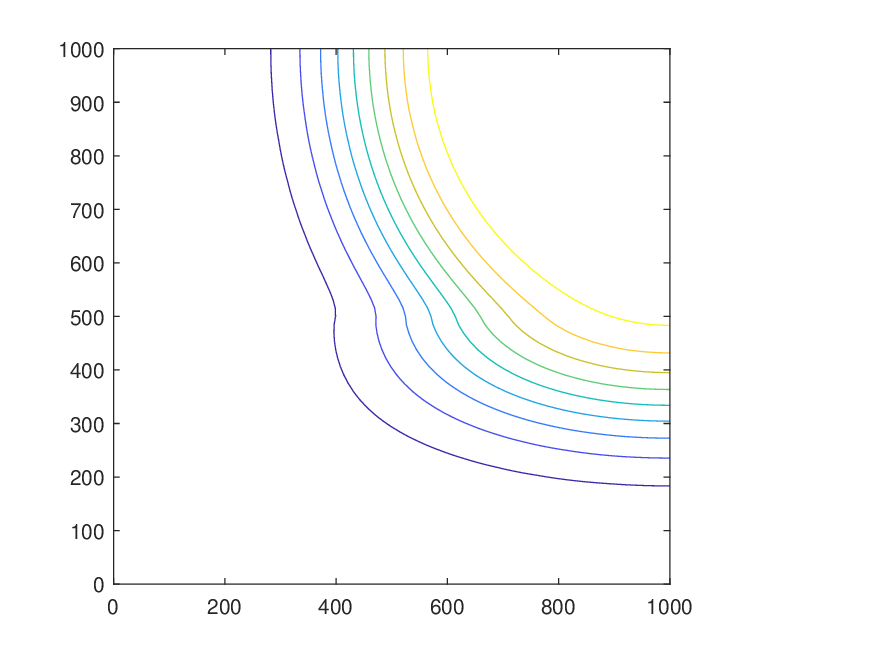}}
		\end{minipage}
		\begin{minipage}[b]{0.23\linewidth}
			{\includegraphics[width=4.5cm]{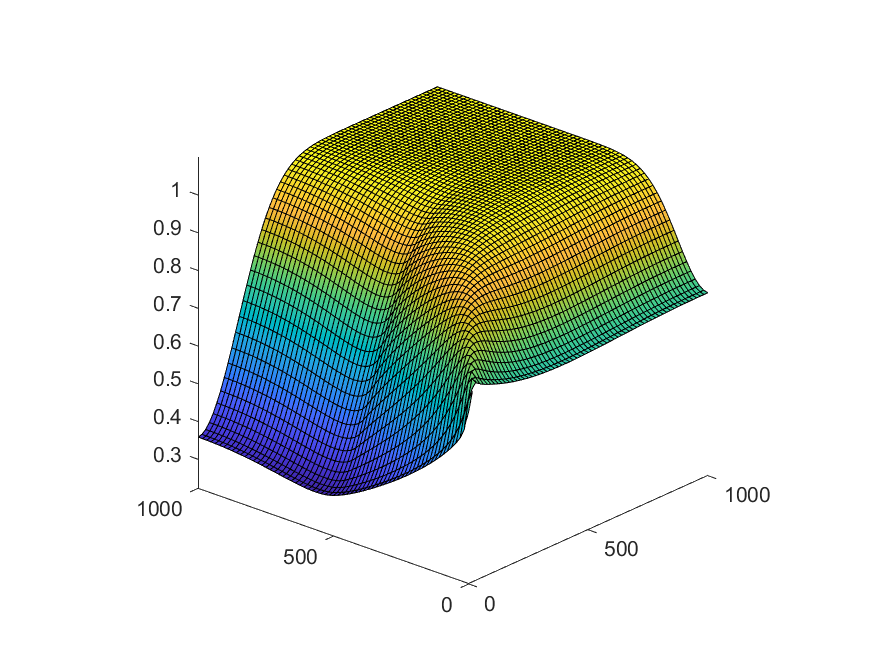}}
		\end{minipage}
		\begin{minipage}[b]{0.23\linewidth}
			{\includegraphics[width=4.2cm]{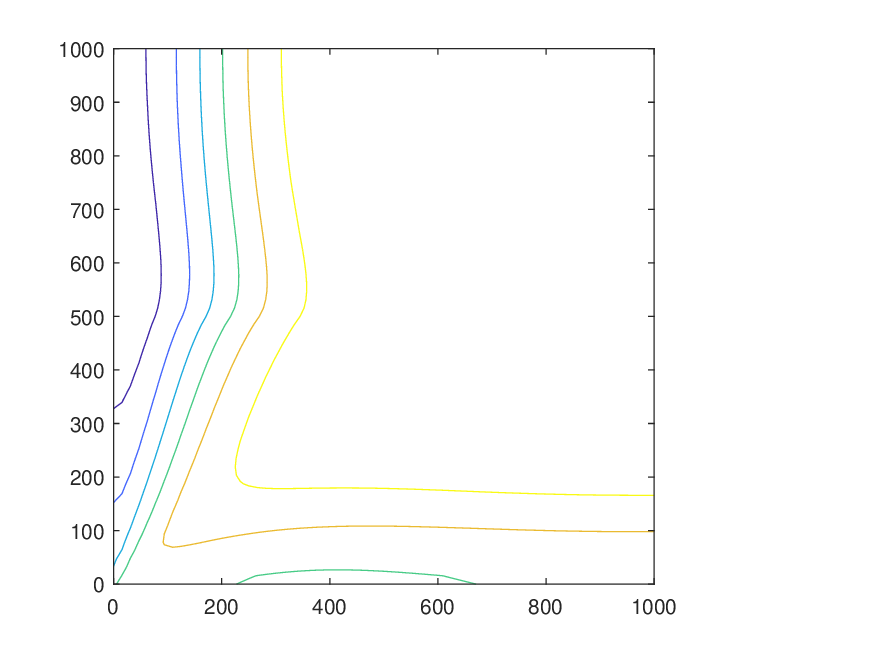}}
		\end{minipage}
		\caption{Surface and contour plots of the concentration in Test 3 at $t=3$ and $t=10$ years for triangular (top) and square (bottom) meshes}\label{fig.test3}
	\end{center}
\end{figure}
\subsubsection*{Test 4} Consider Test 2 with n$k(\bx)=80$ on the sub-domain $\Omega_L$ and$k(\bx)=20$ on the sub-domain $\Omega_U$. The surface and contour plots of the concentration are depicted in Figure~\ref{fig.test4}. As seen in the figure, the concentration front initially moves faster in the vertical direction than in the horizontal direction. This discrepancy arises from the fact that the subdomain $\Omega_L$ possesses a greater permeability, resulting in a higher Darcy velocity than in the subdomain $\Omega_U$. Once the injecting fluid reaches $\Omega_L$, it starts to move much faster in the horizontal direction on $\Omega_L$ than on $\Omega_U$a due to the same reason.
\begin{figure}[h!!]
	\begin{center}
		\begin{minipage}[b]{0.23\linewidth}
			{\includegraphics[width=4.5cm]{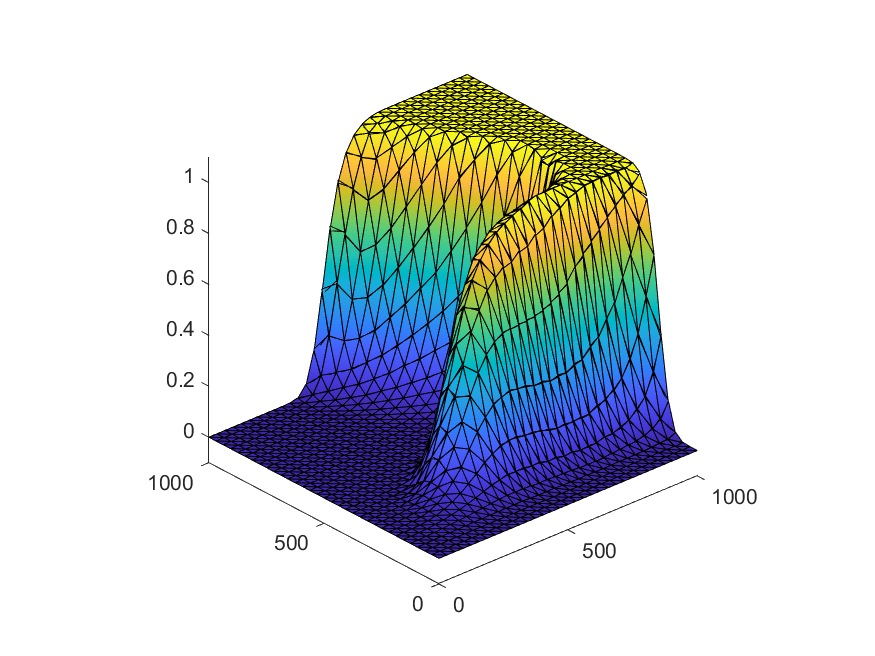}}
		\end{minipage}
		\begin{minipage}[b]{0.2\linewidth}
			{\includegraphics[width=4.2cm]{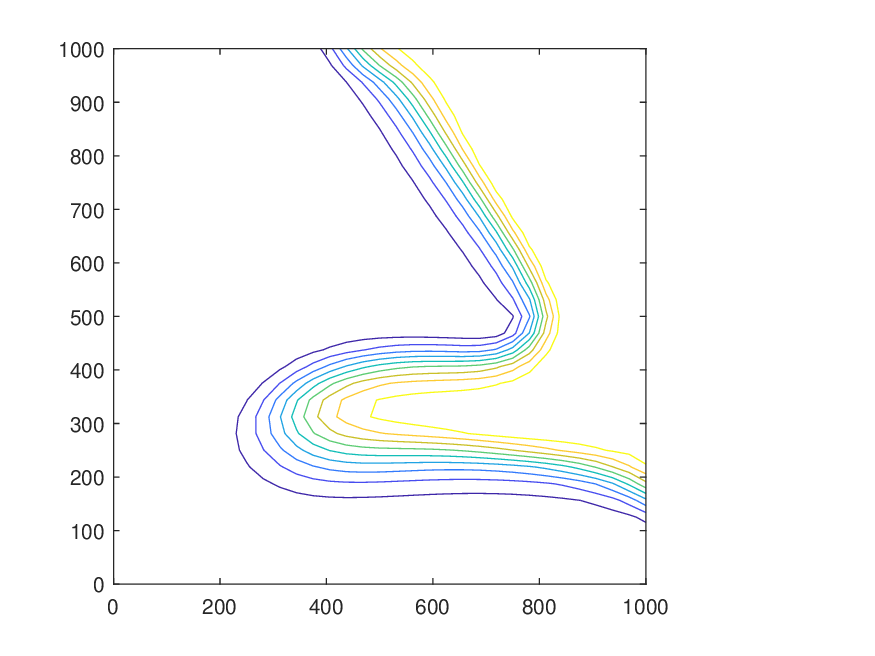}}
		\end{minipage}
		\begin{minipage}[b]{0.23\linewidth}
			{\includegraphics[width=4.5cm]{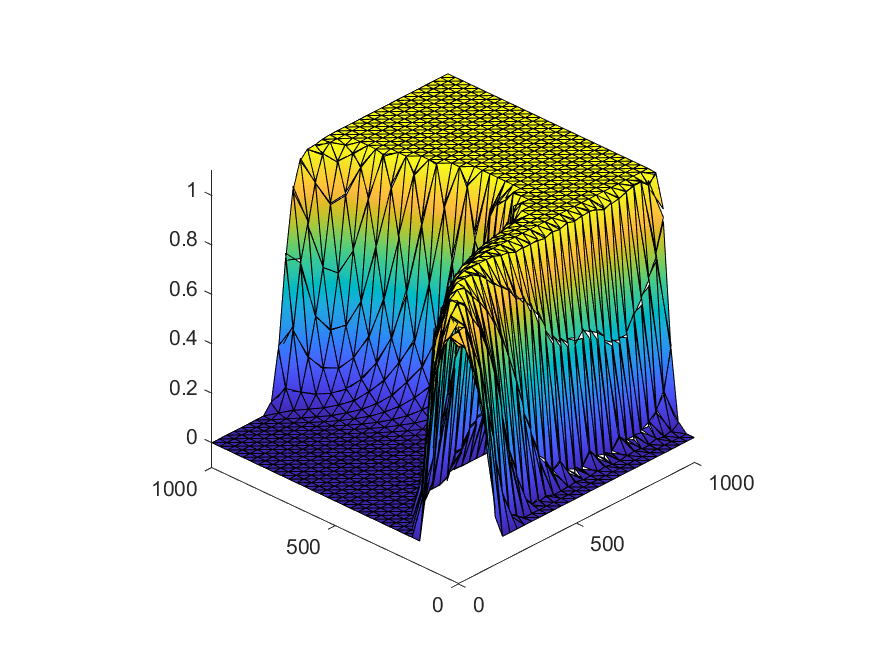}}
		\end{minipage}
		\begin{minipage}[b]{0.23\linewidth}
			{\includegraphics[width=4.2cm]{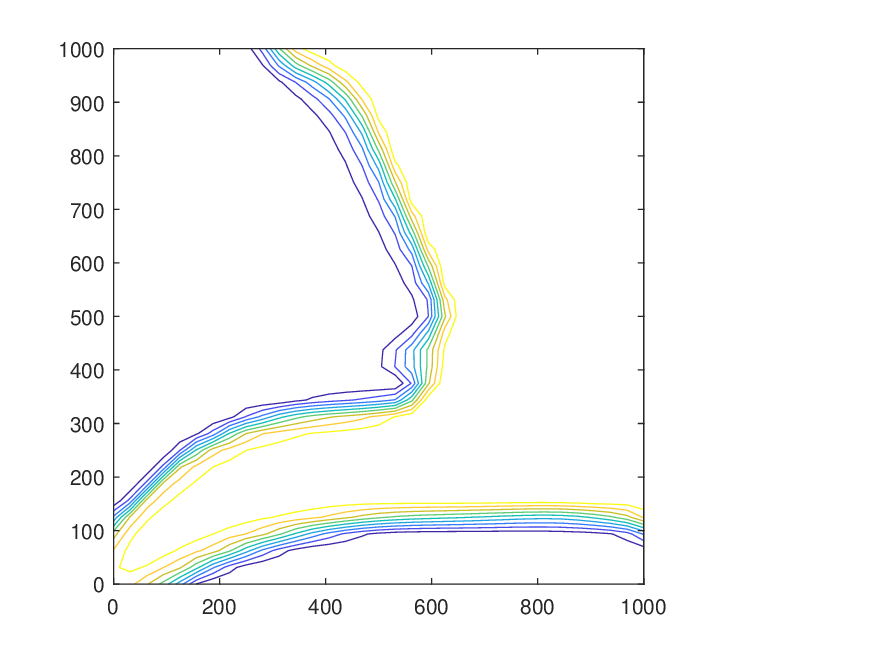}}
		\end{minipage}
		\begin{minipage}[b]{0.23\linewidth}
			{\includegraphics[width=4.5cm]{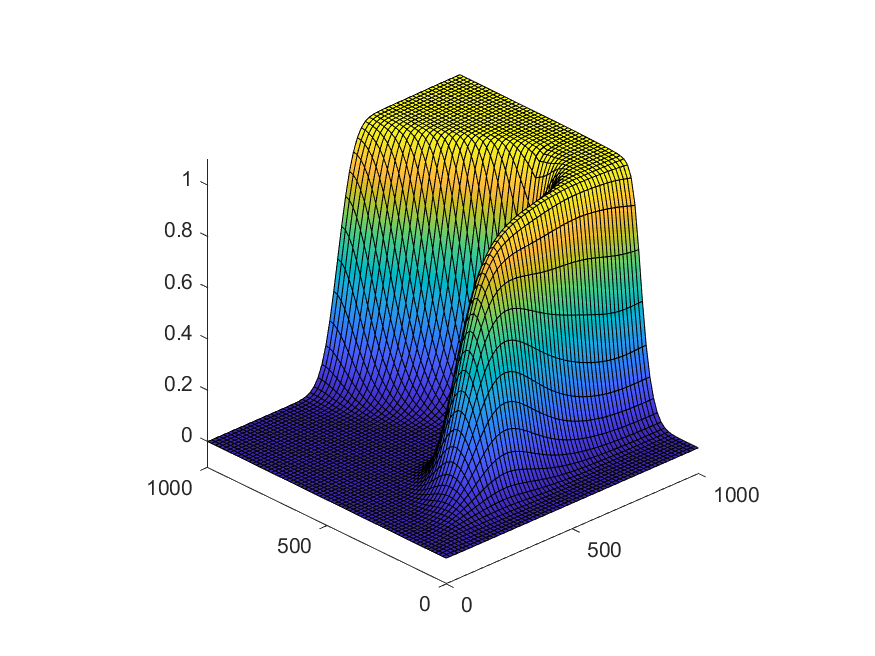}}
		\end{minipage}
		\begin{minipage}[b]{0.2\linewidth}
			{\includegraphics[width=4.2cm]{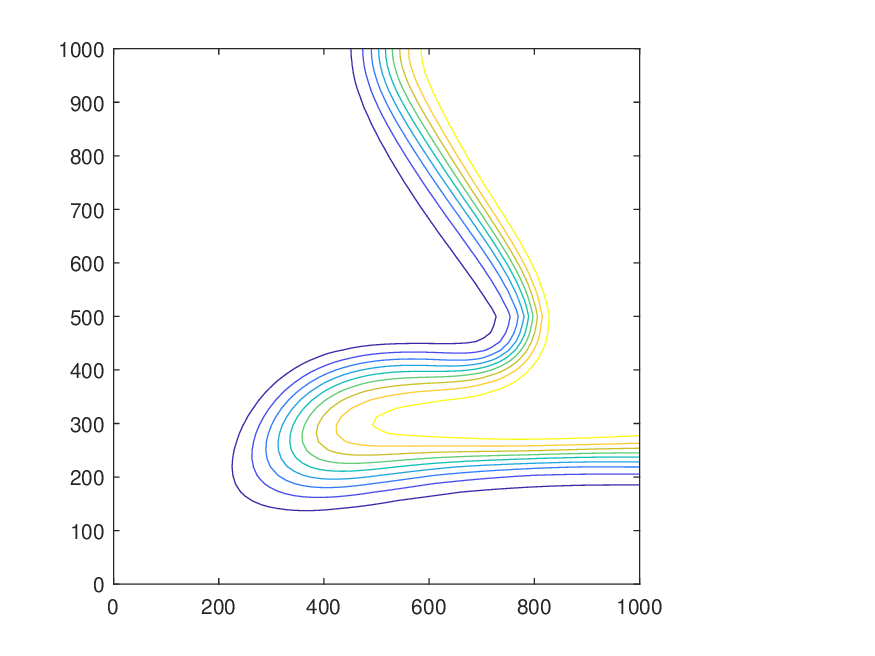}}
		\end{minipage}
		\begin{minipage}[b]{0.23\linewidth}
			{\includegraphics[width=4.5cm]{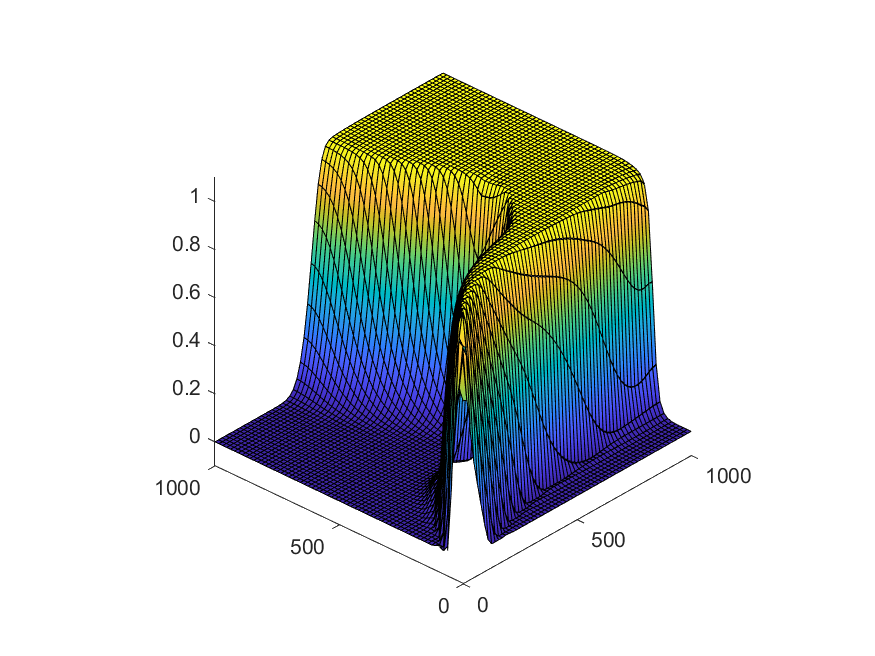}}
		\end{minipage}
		\begin{minipage}[b]{0.23\linewidth}
			{\includegraphics[width=4.2cm]{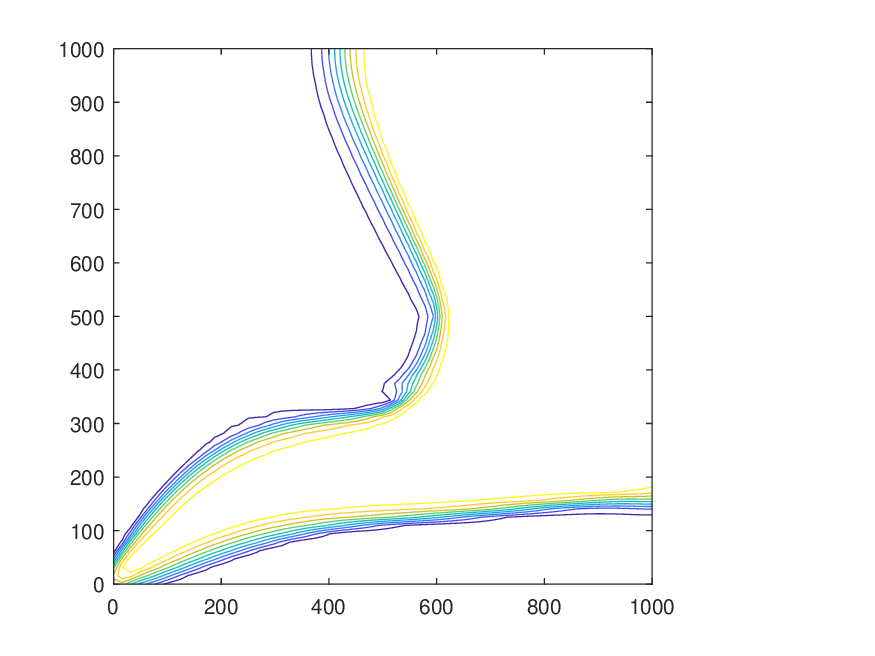}}
		\end{minipage}
		\caption{Surface and contour plots of the concentration in Test 4 at $t=3$ and $t=10$ years for triangular (top) and square (bottom) meshes}\label{fig.test4}
	\end{center}
\end{figure}
\section{Appendix} \label{sec:appendix}
The proof of error estimates for the concentration in Theorem \ref{thm.c} is obtained by modifying the proof of the corresponding results in \cite{Veiga_miscibledisplacement_2021}. For the sake of completeness, we provide here a detailed proof.

\begin{lem}\cite[Lemma 4.4, 4.5]{Veiga_miscibledisplacement_2021}\label{lem.cPcnuuhn}
	Under sufficient smoothness of the continuous data and solution, it holds
	\begin{align}
		(a)&\,\left\|\frac{\partial c^n}{\partial t}-\frac{\P_c c^n-\P_c c^{n-1}}{\tau}\right\|\le \tau^{\half}\left\| \frac{\partial^2 c}{\partial s^2}\right\|_{L^2(t_{n-1},t_n;L^2(\O))}+\tau^{-\half}h^{k+2}\left(\int_{t_{n-1}}^{t_n}\xi_{0,t}^2\ds\right)^{\half},\\
		(b)&\,\|\bu^n-\bu_h^{n-1}\|\lesssim \tau \left\|\frac{\partial \bu}{\partial t}\right\|_{L^\infty(t_{n-1},t_n;L^2(\O))}+\|c^{n-1}-c_h^{n-1}\|+h^{k+1},
	\end{align}
	where $\xi_{0,t}$ is defined in Corollary \ref{cor.ctPcterror}.
\end{lem}
\begin{lem}\cite[(62)]{Veiga_miscibledisplacement_2021}\label{lem.Pinablainfty}
	Under the mesh assumption (\textbf{D3}), for all $c \in H^{k+2}(\cT_h)\cap W^{1,\infty}(\cT_h)$,
	$$\|\bPi_k^{0,K}\nabla \P_c c\|_{0,\infty,K}\lesssim h_K^{-1}\|\bPi_k^{0,K}\nabla \P_c c\|_{0,K}\lesssim h_K^{-1}\|\nabla \P_c c\|_{0,K}\lesssim 1.$$
\end{lem}
\subsection*{Proof of Theorem \ref{thm.c}}

\noindent Throughout this proof, let $\eta$ denote generic constants, which will take different values at different places but will always be independent of $h$ and $\tau$.
	
	\medskip
	
\noindent The triangle inequality with $\P_c c^n$ and Lemma~\ref{lem.cPcerror}.b show
	\begin{equation}\label{eqn.nu1tri}
		\|c^n-c_h^n\|\le \|c^n-\P_c c^n\|+\|\P_c c^n-c_h^n\|\le \eta h^{k+2}+\|\P_c c^n-c_h^n\|.
	\end{equation}
	Let $c_h^n-\P_c c^n=\nu^n \in Z_h$. The discrete fully formulation \eqref{eqn.fullydiscreteq_zh} and the definition of $\P_c c^n$ in \eqref{eqn.projectionP} provide
	\begin{align}
		\cM_h&\left(\frac{\nu^n-\nu^{n-1}}{\tau},\nu^n\right)+\cD_h(\bu_h^{n-1};\nu^n,\nu^n)\\
		&=\cM_h\left(\frac{c_h^n-c_h^{n-1}}{\tau},\nu^n\right)-\cM_h\left(\frac{\P_c c^n-\P_c c^{n-1}}{\tau},\nu^n\right)+\cD_h(\bu_h^{n-1};\nu^n,\nu^n)\\
		&=	(q^{+n}\hc^{n},\nu^n)_h-\Theta_h(\bu_h^{n-1},c_h^{n};\nu^n)-\cM_h\left(\frac{\P_c c^n-\P_c c^{n-1}}{\tau},\nu^n\right)-\cD_h(\bu_h^{n-1};\P_c c^n,\nu^n)\\
		&=(q^{+n}\hc^{n},\nu^n)_h-\Theta_h(\bu_h^{n-1},c_h^{n};\nu^n)-\cM_h\left(\frac{\P_c c^n-\P_c c^{n-1}}{\tau},\nu^n\right)\\
		&\qquad +\cD_h^{\bu^{n}}(\P_c c^n,\nu^n)-\cD_h(\bu_h^{n-1};\P_c c^n,\nu^n)-\cD_h^{\bu^{n}}(\P_c c^n,\nu^n)\\
		&=(q^{+n}\hc^{n},\nu^n)_h-\Theta_h(\bu_h^{n-1},c_h^{n};\nu^n)-\cM_h\left(\frac{\P_c c^n-\P_c c^{n-1}}{\tau},\nu^n\right)\\
		&\qquad +\cD_h^{\bu^{n}}(\P_c c^n,\nu^n)-\cD_h(\bu_h^{n-1};\P_c c^n,\nu^n)-\cD_{\pw}^{\bu^{n}}(c^n,\nu^n)-\Theta_{\pw}^{\bu^{n}}(c^n,\nu^n)-(c^n,\nu^n)\\
		&\qquad +\cN_h(\bu^n;c^n,\nu^n)+\Theta_h^{\bu^{n}}(\P_c c^n,\nu^n)+(\P_c c^n,\nu^n)_h.	\label{eqn1}
	\end{align}
	An integration by parts yields
	\begin{align}
		-\cD_{\pw}^{\bu^{n}}(c^n,\nu^n)+\sum_{e \in \cE_h}\int_e(D(\bu^n)\nabla c^n \cdot \bn_e )\jump{\nu^n}\ds&=(\divc(D(\bu^n)\nabla c^n),\nu^n)
	\end{align}
	and
	\begin{align}
		-\Theta_{\pw}^{\bu^{n}}(c^n,\nu^n)-\sum_{e \in \cE_h}\int_e\frac{c^n\bu^n \cdot n_e}{2} \jump{\nu^n}\ds&=-(\bu^n\cdot \nabla c^n,\nu^n)-(q^{+n}c^n,\nu^n).
	\end{align}
	Lemma~\ref{lem.propertiesdiscrete}.d reads $\cD_h(\bu_h^{n-1};\nu^n,\nu^n) \ge D_* |\nu^n|_{1,\cT_h}^2$ where $D_*$ is independent of $h$ and $\bu_h^{n-1}$. These estimates in \eqref{eqn1} together with the definition of $\cN_h(\bu^n;c^n,\nu^n)$ and $\divc(D(\bu^n)\nabla c^n)-\bu^n\cdot \nabla c^n-q^{+n}c^n=\phi \frac{\partial c^n}{\partial t}-q^{+n}\hc^n$ from \eqref{eqn.model} show
	\begin{align}
		\cM_h&\left(\frac{\nu^n-\nu^{n-1}}{\tau},\nu^n\right)+D_*|\nu^n|_{1,\cT_h}^2\\
		\le &\left(\cM\left(\phi \frac{\partial c^n}{\partial t},\nu^n\right)-\cM_h\left(\frac{\P_c c^n-\P_c c^{n-1}}{\tau},\nu^n\right)\right)+(\Theta_h^{\bu^{n}}(\P_c c^n,\nu^n)-\Theta_h(\bu_h^{n-1},c_h^{n};\nu^n)) \\
		&\qquad +(\cD_h^{\bu^{n}}(\P_c c^n,\nu^n)-\cD_h(\bu_h^{n-1};\P_c c^n,\nu^n)) +((\P_c c^n,\nu^n)_h-(c^n,\nu^n))\\
		&\qquad+((q^{+n}\hc^{n},\nu^n)_h-(q^{+n}\hc^n,\nu^n))=:A_1+A_2+A_3+A_4+A_5.\label{eqna1a2a3a4a5}
	\end{align}
	The remaining arguments follow from the proof of \cite[Theorem 2]{Veiga_miscibledisplacement_2021}. However, for the sake of completeness, we provide a proof.
	
	\smallskip

	\noindent\textbf{Step 1: estimation of $A_1$.} The definition of $\cM(\bullet,\bullet)$ and $\cM_h(\bullet,\bullet)$ in \eqref{defn.bilinear} and \eqref{defn.Mh}, orthogonality and continuity properties of $\Pi_{k+1}^{0,K}$, Cauchy Schwarz inequality, $S_\cM^K(z_h,\tilde{z}_h)\le M_1^\cM\|z_h\|_{0,K}\|\tilde{z}_h\|_{0,K}$ from \eqref{eqn.SM} prove
	\begin{align}
		A_1
		&=\sum_{K \in \cT_h}\bigg[\left(\phi \frac{\partial c^n}{\partial t},\nu^n\right)_{0,K}-\left(\phi \Pi_{k+1}^{0,K}\left(\frac{\P_c c^n-\P_c c^{n-1}}{\tau}\right),\Pi_{k+1}^{0,K}\nu^n\right)_{0,K}\\
		&\qquad - \nu_\cM^K(\phi) S_\cM^K\left((I-\Pi_{k+1}^{0,K})\left(\frac{\P_c c^n-\P_c c^{n-1}}{\tau}\right),(I-\Pi_{k+1}^{0,K})\nu^n\right)\bigg]\\
		&=\sum_{K \in \cT_h}\bigg[\left(\phi \frac{\partial c^n}{\partial t},\nu^n\right)_{0,K}-\left(\Pi_{k+1}^{0,K}\left(\phi \Pi_{k+1}^{0,K}\left(\frac{\P_c c^n-\P_c c^{n-1}}{\tau}\right)\right),\nu^n\right)_{0,K}\\
		&\qquad - \nu_\cM^K(\phi) S_\cM^K\left((I-\Pi_{k+1}^{0,K})\left(\frac{\P_c c^n-\P_c c^{n-1}}{\tau}\right),(I-\Pi_{k+1}^{0})\nu^n\right)\bigg]\\
		&\le \eta\bigg[\left\|\phi \frac{\partial c^n}{\partial t}-\Pi_{k+1}^{0}\left(\phi \Pi_{k+1}^{0}\left(\frac{\P_c c^n-\P_c c^{n-1}}{\tau}\right)\right)\right\|\| \nu^n\|\\
		&\qquad + \left\|(I-\Pi_{k+1}^{0})\left(\frac{\P_c c^n-\P_c c^{n-1}}{\tau}\right)\right\|\|\nu^n\|\bigg]=:\eta(A_{1,1}+A_{1,2})\|\nu^n\|.\label{eqna1}
	\end{align}
	The continuity of the $L^2$ projector $\Pi_{k+1}^{0}$, boundedness of $\phi$, and Lemma \ref{lem.approx} read
	\begin{align}
		A_{1,1}&\le \left\|(I-\Pi_{k+1}^{0})\phi \frac{\partial c^n}{\partial t}\right\|+\left\|\Pi_{k+1}^{0}\left(\phi \frac{\partial c^n}{\partial t}-\phi\Pi_{k+1}^{0} \frac{\partial c^n}{\partial t}\right)\right\| \\
		&\qquad +\left\|\Pi_{k+1}^{0}\left(\phi \Pi_{k+1}^{0}\left(\frac{\partial c^n}{\partial t}-\frac{\P_c c^n-\P_c c^{n-1}}{\tau}\right)\right)\right\|\\
		&\le \eta \bigg[h^{k+2}\left(\left|\phi \frac{\partial c^n}{\partial t}\right|_{k+2,\cT_h}+\left|\frac{\partial c^n}{\partial t}\right|_{k+2,\cT_h}\right)+\left\| \frac{\partial c^n}{\partial t}-\frac{\P_c c^n-\P_c c^{n-1}}{\tau}\right\|\bigg].\label{eqna11}
	\end{align} 
	Analogous arguments provides
	\begin{align}
		A_{2,1}&= \left\|(I-\Pi_{k+1}^{0})\left(\frac{\P_c c^n-\P_c c^{n-1}}{\tau}-\frac{\partial c^n}{\partial t}\right)\right\|+\left\|(I-\Pi_{k+1}^{0})\frac{\partial c^n}{\partial t}\right\|\\
		&\le \eta\bigg[ \left\|\frac{\P_c c^n-\P_c c^{n-1}}{\tau}-\frac{\partial c^n}{\partial t} \right\|+h^{k+2}\left|\frac{\partial c^n}{\partial t}\right|_{k+2,\cT_h}\bigg].
	\end{align}
	This and \eqref{eqna11} in \eqref{eqna1} together with Lemma \ref{lem.cPcnuuhn}.a result in
	\begin{align}
		A_1&\le \eta \bigg(h^{k+2}+\tau^{-\half}h^{k+2}\left(\int_{t_{n-1}}^{t_n}\xi_{0,t}^2\ds\right)^{\half}+\tau^\half\left\|\frac{\partial^2 c}{\partial s^2}\right\|_{L^2(t_{n-1},t_n;L^2(\O))}\bigg)\|\nu^n\|.\label{eqna1new}
	\end{align}
	\noindent\textbf{Step 2: estimation of $A_2$.} The definition of $\Theta_h^{\bu^{n}}(\bullet,\bullet)$ and $\Theta_h(\bu_h^{n-1},\bullet;\bullet)$ in \eqref{eqn.Thetahu}, and \eqref{eqn.Thetah} and $((q^{+n}+q^{-n})\nu^n,\nu^n)_h \ge 0$ show
	\begin{align}
		A_2
		&=\half[\sum_{K\in \cT_h}(\bu^n\cdot \bPi_{k}^{0,K}(\nabla \P_cc^n),\Pi_{k+1}^{0,K}\nu^n)_{0,K}-(\bu_h^{n-1}\cdot\nabla c_h^n,\nu^n)_h-
		((q^{+n}+q^{-n})\nu^n,\nu^n)_h\\
		&\qquad -\sum_{K\in \cT_h}(\bu^n,\Pi_{k+1}^{0,K} \P_c c^n\cdot\bPi_{k}^{0,K}(\nabla \nu^n))_{0,K}+(\bu_h^{n-1}, c_h^n\nabla \nu^n)_h]\\
		&\le \half[\sum_{K\in \cT_h}(\bu^n\cdot \bPi_{k}^{0,K}(\nabla \P_cc^n),\Pi_{k+1}^{0,K}\nu^n)_{0,K}-(\bu_h^{n-1}\cdot\nabla c_h^n,\nu^n)_h\\
		&\qquad -\sum_{K\in \cT_h}(\bu^n,\Pi_{k+1}^{0,K} \P_c c^n\cdot\bPi_{k}^{0,K}(\nabla \nu^n))_{0,K}+(\bu_h^{n-1}, c_h^n\nabla \nu^n)_h]. \label{eqna2}
	\end{align}
	Since $c_h^n-\P_c c^n=\nu^n$,
	\begin{align}
		0&=(\bu_h^{n-1}\cdot\nabla \nu^n,\nu^n)_h-(\bu_h^{n-1}\cdot\nabla \nu^n,\nu^n)_h\\
		&=(\bu_h^{n-1}\cdot\nabla c_h^n,\nu^n)_h-(\bu_h^{n-1}\cdot\nabla \P_c c^n,\nu^n)_h-(\bu_h^{n-1}\cdot\nabla \nu^n,c_h^n)_h+(\bu_h^{n-1}\cdot\nabla \nu^n,\P_c c^n)_h.
	\end{align}
	Hence, \eqref{eqna2} becomes
	\begin{align}
		A_2&\le  \half[\sum_{K\in \cT_h}(\bu^n\cdot \bPi_{k}^{0,K}(\nabla \P_cc^n),\Pi_{k+1}^{0,K}\nu^n)_{0,K}-(\bu_h^{n-1}\cdot\nabla \P_c c^n,\nu^n)_h\\
		&\qquad +(\bu_h^{n-1}, \P_c c^n\nabla \nu^n)_h-\sum_{K\in \cT_h}(\bu^n,\Pi_{k+1}^{0,K} \P_c c^n\cdot\bPi_{k}^{0,K}(\nabla \nu^n))_{0,K}]\\
		&= \half[\sum_{K\in \cT_h}((\bu^n-\bPi_{k}^{0,K}\bu^n)\cdot \bPi_{k}^{0,K}(\nabla \P_cc^n),\Pi_{k+1}^{0,K}\nu^n)_{0,K}+((\bu^n-\bu_h^{n-1})\cdot\nabla \P_c c^n,\nu^n)_h\\
		&\qquad -(\bu^n-\bu_h^{n-1}, \P_c c^n\nabla \nu^n)_h-\sum_{K\in \cT_h}(\bu^n-\bPi_{k}^{0,K}\bu^n,\Pi_{k+1}^{0,K} \P_c c^n\cdot\bPi_{k}^{0,K}(\nabla \nu^n))_{0,K}].
	\end{align}
	A triangle inequality with $\P_c c^{n-1}$ and Lemma~\ref{lem.cPcerror}.b lead to  $\|c^{n-1}-c_h^{n-1}\|\le \eta h^{k+2}+\|\nu^{n-1}\|$. Hence, Lemma~\ref{lem.cPcnuuhn}.b reads \begin{equation}\label{eqn.unuhn1}
		\|\bu^n-\bu_h^{n-1}\|\le \eta( \tau +h^{k+1}+\|\nu^{n-1}\|).
	\end{equation}
	This, a generalised \Holder inequality, Lemma~\ref{lem.Pinablainfty}, the approximation property of Lemma~\ref{lem.approx}, and the continuity of projection operator imply
	\begin{equation}
		A_2\le \eta(\tau +h^{k+1}+\|\nu^{n-1}\|)(\|\nu^n\|+|\nu^n|_{1,\cT_h}).\label{eqna2new}
	\end{equation}
	\noindent\textbf{Step 3: estimation of $A_3$.} A simple manipulation leads to 
		\begin{align}
		A_3
		&=\sum_{K \in \cT_h}(( (D(\bu^n)-D(\bPi_k^{0,K}\bu^n)) \bPi_{k}^{0,K}(\nabla \P_c c^n),\bPi_{k}^{0,K}(\nabla \nu^n))_{0,K}\\
		&\qquad -( (D(\bPi_k^{0,K}\bu^n) -D(\bPi_k^{0,K}\bu_h^{n-1}))\bPi_{k}^{0,K}(\nabla \P_c c^n),\bPi_{k}^{0,K}(\nabla \nu^n))_{0,K})\\
		&\qquad + (\nu_\cD^K(\bu^n)-\nu_\cD^K(\bu_h^{n-1}))S_\cD^K((I-\Pi_{k+1}^{\nabla,K})\P_c c^n,(I-\Pi_{k+1}^{\nabla,K})\nu^n).
	\end{align}
	The generalised \Holder inequality, Lipschitz continuity of $\cD(\bullet,\bullet,\bullet)$, Lemma~\ref{lem.approx}, Lemma~\ref{lem.Pinablainfty}, continuity of the projection operator, the definition of $\nu_\cD^K(\bullet)$ and the stability property of $S_\cD^K(\bullet,\bullet)$ show $A_3 \le \eta (h^{k+1}+\|\bu^n-\bu_h^{n-1}\|)|\nu^n|_{\cT_h}.$
	Consequently, \eqref{eqn.unuhn1} provides
	\begin{align}
		A_3 \le \eta( \tau +h^{k+1}+\|\nu^{n-1}\|)|\nu^n|_{1,\cT_h}.
	\end{align}
	\noindent\textbf{Step 4: estimation of $A_4$ and $A_5$.} An introduction of $\Pi_{k+1}^0c^n$, Lemma~\ref{lem.cPcerror}.b, and Lemma~\ref{lem.approx} show
	\begin{align}
		A_4&=(\Pi_{k+1}^0(\P_c c^n-c^n),\nu^n)_h-((I-\Pi_{k+1}^0)c^n,\nu^n)\le \eta h^{k+2}\|\nu^n\|.\label{eqna4}
	\end{align}
	Lemma~\ref{lem.approx} yields
	\begin{align}
		A_5&=((I-\Pi_{k+1}^{0})q^{+n}\hc^{n},\nu^n)_h \le \eta h^{k+2}|q^{+n}\hc^{n}|_{k+2,\cT_h}\|\nu^n\|=\eta h^{k+2}\|\nu^n\|.\label{eqna5}
	\end{align}
\noindent	\textbf{Step 5: conclusion.} A combination of the estimates in $A_1-A_5$ in \eqref{eqn.a1a2a3a4a5} leads to
	\begin{align}
		\cM_h\left(\frac{\nu^n-\nu^{n-1}}{\tau},\nu^n\right)+	\frac{\|\nu^n\|^2}{\tau}+D_*|\nu^n|_{1,\cT_h}^2
		&\le \|\nu^n\|(\omega_3^n+\|\nu^{n-1}\|\omega_1^n)\\
		&\qquad + |\nu^n|_{1,\cT_h}(\omega_4^n+\|\nu^{n-1}\|\omega_2^n)\label{eqnmh}
	\end{align}
	where
	\[\omega_i^n\le \eta,\,i=1,2,\quad\omega_3^n\le \eta(h^{k+1}+\tau^{-\half}h^{k+2}R_1+\tau^{\half}R_2+\tau), \quad \omega_4^n \le \eta(\tau +h^{k+1})\]
	wirh $$R_1^2:=\int_{t_{n-1}}^{t_n}\xi_{0,t}^2\ds,\,\quad R_2:=\|\frac{\partial^2 c}{\partial s^2}\|_{L^2(t_{n-1},t_n;L^2(\O))}.$$
	Define the discrete norm, for all $w_h \in Z_h$,
	\[\|w_h\|_{0,h}:=\cM_h(w_h,w_h).\]
	Then, Lemma~\ref{lem.propertiesdiscrete}.b implies that there exists positive constants $c_*$ and $c^*$ independent of $h$ such that
	\begin{equation}\label{eqn.normh}
		c_*\|w_h\|_{0,h}\le \|w_h\|\le c^*\|w_h\|_{0,h}.
	\end{equation}
	Therefore, \eqref{eqnmh} results in
	\begin{align}
		\|\nu^n\|_{0,h}^2+\tau D_*|\nu^n|_{1,\cT_h}^2
		&\le	\cM_h(\nu^{n-1},\nu^n)+\tau\|\nu^n\|_{0,h}(c^*\omega_3^n+(c^{*})^{2}\omega_1^n\|\nu^{n-1}\|_{0,h})\\
		&\qquad +\tau |\nu^n|_{1,\cT_h}(\omega_4^n+c^*\omega_2^n\|\nu^{n-1}\|_{0,h})=:T_1+T_2+T_3.\label{eqn.t1t2t3}
	\end{align}
	The scaling property of $\cM_h(\bullet,\bullet)$, the definition of $\|\bullet\|_{0,h}$, and an application of Young's inequality show
	\begin{align}
		T_1+T_2&\le \|\nu^n\|_{0,h}((1+\tau \eta)\|\nu^{n-1}\|_{0,h}+\tau c^*\omega_3^n)\\
		& \le \half(\|\nu^n\|_{0,h}^2+((1+\tau \eta)\|\nu^{n-1}\|_{0,h}+\tau c^*\omega_3^n)^2).
	\end{align}
	Young's inequality provides
	\begin{align*}
		T_3 &\le \tau D_*|\nu^n|_{1,\cT_h}^2+\frac{\tau}{4D_*}(\omega_4^n+\eta\|\nu^{n-1}\|_{0,h})^2\le \tau D_*|\nu^n|_{1,\cT_h}^2+\frac{\tau}{2}\eta((\omega_4^{n})^2+\|\nu^{n-1}\|_{0,h}^2).
	\end{align*}
	A substitution of $T_1$-$T_3$ in \eqref{eqn.t1t2t3} leads to
	\begin{equation}
		\|\nu^n\|_{0,h}^2 \le ((1+\tau \eta)\|\nu^{n-1}\|_{0,h}+\tau c^*\omega_3^n)^2+\tau\eta((\omega_4^{n})^2+\|\nu^{n-1}\|_{0,h}^2).\label{eqnnun}
	\end{equation}
	An application of Youngs inequality implies
	\begin{align}
		((1+\tau \eta)\|\nu^{n-1}\|_{0,h}+\tau c^*\omega_3^n)^2&=	((1+\tau \eta)^2\|\nu^{n-1}\|_{0,h}^2+2\tau^\half\|\nu^{n-1}\|_{0,h}	\tau^\half(1+\tau \eta)\tau c^*\omega_3^n +(\tau c^*\omega_3^n)^2\\
		&\le ((1+\tau \eta)^2+\tau)\|\nu^{n-1}\|_{0,h}^2+(\tau(1+\tau \eta)^2+\tau^2)(c^*\omega_3^n)^2\\
		&\le (1+\tau \eta)\|\nu^{n-1}\|_{0,h}^2+\tau \eta(\omega_3^n)^2.
	\end{align}
	This in \eqref{eqnnun} yields
	\begin{equation}
		\|\nu^n\|_{0,h}^2\le (1+\tau \eta)\|\nu^{n-1}\|_{.0,h}^2+\tau \eta[(\omega_3^n)^2+(\omega_4^n)^2]. 
	\end{equation}
	Hence, recursive process and the equivalence relation in \eqref{eqn.normh} prove
	\begin{equation}
		\|\nu^n\|_{}^2\le (1+\tau \eta)\|\nu^{0}\|_{}^2+\tau \eta \sum_{j=1}^n \gamma_j,\label{eqn.nunnu0}
	\end{equation}
	where \[\gamma_j:=(\omega_3^j)^2+(\omega_4^j)^2 \,\mbox{ and  }\, n\le T/\tau.\]
	A triangle inequality with $c^0$ and Lemma \ref{lem.cPcerror}.b leads to
	\begin{equation}
		\|\nu^{0}\|_{}=\|c_{0,h}-\P_c c^0\|_{} \le \|c_{0,h}-c^0\|_{}+\eta h^{k+2}.\label{eqn.nu0}
	\end{equation}
	The definition of $\gamma^j$, $\omega_3^j$, and $\omega_4^j$ shows
	\begin{align}
		\tau \eta \sum_{j=1}^n \gamma_j&\le \eta \sum_{j=1}^n (\tau(\omega_3^j)^2+\tau(\omega_4^j)^2)\\
		&\le \eta \sum_{j=1}^n \left(\tau\left(h^{k+1}+\tau^{-\half}h^{k+2}\left(\int_{t_{j-1}}^{t_j}\xi_{0,t}^2\ds\right)^\half+\tau^{\half}\left\|\frac{\partial^2 c}{\partial s^2}\right\|_{L^2(t_{j-1},t_j;L^2(\O))}+\tau\right)^2+\tau(\tau +h^{k+1})^2\right)\\
		&\le \eta\left(\sum_{j=1}^n\tau(\tau +h^{k+1})^2+(h^{k+2})^2\sum_{j=1}^n\int_{t_{j-1}}^{t_j}\xi_{0,t}^2\ds+\tau^2\sum_{j=1}^n\left\|\frac{\partial^2 c}{\partial s^2}\right\|_{L^2(t_{j-1},t_j;L^2(\O))}^2\right)\\
		&\le \eta\left((\tau +h^{k+1})^2+(h^{k+2})^2\int_{0}^{t_n}\xi_{0,t}^2\ds+\tau^2\int_{0}^{t_n}\left\|\frac{\partial^2 c}{\partial s^2}\right\|_{}^2\ds\right)\label{eqn.gammajbound}
	\end{align}
	with the relation $\sum_{j=1}^n \tau \le T$in the last step. A combination of \eqref{eqn.gammajbound} and \eqref{eqn.nu0} in \eqref{eqn.nunnu0} leads to
	\[	\|\nu^n\|_{}\le \eta (\|c_{0,h}-c^0\|+h^{k+1}+\tau).\]
	This and \eqref{eqn.nu1tri} concludes the proof.\qed

\medskip

\noindent {\bf{Acknowledgements.}} The first author thanks the Department of Science and Technology (DST-SERB), India, for supporting this work through the core research grant CRG/2021/002410.  The second author thanks Indian Institute of Space Science and Technology (IIST) for the financial support towards the research work.

\subsection*{Declarations}

\noindent {\bf Conflict of Interest.} The authors declare that they have no conflict of interest.
\bibliographystyle{amsplain}
\bibliography{VEMBib}
\end{document}